\documentclass[11pt]{article}
\usepackage{amsmath, amsthm, amssymb, psfrag,verbatim}
\usepackage{hyperref}

\newcommand{\be}{\begin{equation}}
\newcommand{\ee}{\end{equation}}
\newcommand{\bea}{\begin{eqnarray}}
\newcommand{\eea}{\end{eqnarray}}
\newcommand{\beas}{\begin{eqnarray*}}
\newcommand{\eeas}{\end{eqnarray*}}

\newcommand{\bbD}{\mathbb D}
\newcommand{\bbE}{\mathbb E}

\newcommand{\bbH}{\mathbb H}

\newcommand{\bbL}{\mathbb L}

\newcommand{\bbP}{\mathbb P}

\newcommand{\bbR}{\mathbb R}
\newcommand{\bbS}{\mathbb S}

\newcommand{\E}[1]{\mathbb{E} \left[#1\right]}

\newcommand{\scF}{\mathcal F}

\newcommand{\scL}{\mathcal L}

\newcommand{\scN}{\mathcal N}
\newcommand{\scO}{\mathcal O}

\newcommand{\scS}{\mathcal S}

\DeclareMathOperator*{\esssup}{ess\,sup}

\newcommand{\ang}[1]{\ensuremath{ \left \langle #1 \right \rangle }}
\newcommand{\crl}[1]{\ensuremath{ \left\{ #1 \right\} }}

\newcommand{\brak}[1]{\ensuremath{\left( #1 \right)}}

\newcommand{\veps}{\varepsilon}
\newcommand{\half}{\frac{1}{2}}
\newcommand{\p}{\mathbb{P}}

\newcommand{\norm}[1]{\ensuremath{\left\| #1 \right\|}}
\newcommand{\abs}[1]{\ensuremath{\left| #1 \right|}}

\newtheorem{theorem}{Theorem}[section]
\newtheorem{definition}[theorem]{Definition}
\newtheorem{proposition}[theorem]{Proposition}
\newtheorem{corollary}[theorem]{Corollary}
\newtheorem{lemma}[theorem]{Lemma}
\newtheorem{remark}[theorem]{Remark}
\newtheorem{example}[theorem]{Example}
\newtheorem{examples}[theorem]{Examples}
\newtheorem{foo}[theorem]{Remarks}
\newenvironment{Example}{\begin{example}\rm}{\end{example}}

\newenvironment{Remark}{\begin{remark}\rm}{\end{remark}}

\title{BSDEs with terminal conditions that have\\ bounded Malliavin derivative}
\author{Patrick Cheridito\footnote{Supported by NSF Grant DMS-0642361.}\\
ORFE, Sherrerd Hall 204\\
Princeton University\\
Princeton, NJ 08544, USA
\and
Kihun Nam\\
PACM, Fine Hall 210\\
Princeton University\\
Princeton, NJ 08544, USA
}
\date{}

\begin{document}
\maketitle

\begin{abstract}
We show existence and uniqueness of solutions to BSDEs of the form
$$
Y_t = \xi + \int_t^T f(s,Y_s,Z_s)ds - \int_t^T Z_s dW_s$$
in the case where the terminal condition $\xi$ has bounded Malliavin derivative. 
The driver $f(s,y,z)$ is assumed to be Lipschitz continuous in $y$ but only locally Lipschitz continuous in 
$z$. In particular, it can grow arbitrarily fast in $z$. If in addition to having bounded 
Malliavin derivative, $\xi$ is bounded, the driver 
needs only be locally Lipschitz continuous in $y$. In the special case where the 
BSDE is Markovian, we obtain existence and uniqueness results for 
semilinear parabolic PDEs with non-Lipschitz nonlinearities.
We discuss the case where there is no lateral boundary as well as 
lateral boundary conditions of Dirichlet and Neumann type.\\[2mm]
\end{abstract}

\setcounter{equation}{0}
\section{Introduction}
\label{sec:intro}

We study BSDEs (backward stochastic differential equations) of the form
\begin{equation}\label{bsde}
Y_t = \xi + \int_t^T f(s,Y_s,Z_s)ds - \int_t^T Z_s dW_s,
\end{equation}
for an $n$-dimensional Brownian motion $W$ and a terminal condition
$\xi$ that is measurable with respect to the sigma-algebra generated by $W$.
If the driver $f(s,y,z)$ is Lipschitz in $(y,z)$, it can be shown with a Picard iteration argument that
\eqref{bsde} has a unique solution for any square-integrable terminal condition $\xi$; 
see Pardoux and Peng \cite{PP90}. Kobylanski \cite{Kobylanski} proved the existence
of a unique solution in the case where $f$ does not grow faster than quadratically in $z$ and $\xi$ is bounded.
BSDEs with drivers of quadratic growth in $z$ and unbounded terminal conditions have been studied by
Briand and Hu \cite{BH, BH2} as well as Delbaen et al. \cite{DHR}. Delbaen et al. \cite{DHB} showed that if the driver
$f$ only depends on $z$, is convex and has superquadratic growth, 
there exist bounded terminal conditions such that the BSDE
\eqref{bsde} has no solution with bounded $Y$, and if the BSDE admits a solution with bounded $Y$, 
it has infinitely many of them. Moreover, they proved the existence of a solution 
for Markovian BSDEs when the terminal value is a bounded continuous function of the
terminal value of a forward process. Richou \cite{Richou} derived the existence of solutions to more general
Markovian BSDEs in the case where $f$ and $\xi$ satisfy a local Lipschitz condition with respect to the underlying 
forward process. In Cheridito and Stadje \cite{CS} it is shown that BSDEs whose drivers are 
convex in $z$ have unique solutions with bounded $Z$ if $f$ and $\xi$ are Lipschitz continuous functionals
of the path of the underlying Brownian motion.

In this paper $f$ can grow arbitrarily fast in $z$, and we do not make Markov or convexity assumptions.
On the other hand, we require $f$ and $\xi$ to be Malliavin differentiable. Recently, Malliavin calculus 
has also been applied in the study of BSDEs by Hu et al. \cite{HNS} and Briand and Elie \cite{BE}.
In Section \ref{sec:general} we show that if $\xi$ has bounded Malliavin derivative, 
\eqref{bsde} has a unique solution for drivers $f$ that are Lipschitz in $y$ and locally Lipschitz in $z$.
If $\xi$ is also bounded, $f$ only needs to be locally Lipschitz in $y$.
In Section \ref{sec:Lipschitz} we show that every terminal condition that is Lipschitz in 
the underlying Brownian motion has a bounded Malliavin derivative. On the other 
hand, we give an example of a terminal condition with bounded Malliavin derivative 
that is not Lipschitz in the underlying Brownian motion. This shows that our condition is
weaker than Lipschitz continuity in the underlying Brownian motion.
In Sections \ref{sec:pde}--\ref{sec:pdeNeu} we generalize results on the relation between Markovian
BSDEs and semilinear parabolic PDEs to the case of non-Lipschitz nonlinearities. 
In Section \ref{sec:pde} we study Markovian BSDEs based on forward processes 
following standard diffusion dynamics and related PDEs for functions 
$u : [0,T] \times \mathbb{R}^m \to \mathbb{R}$. The general results of Section \ref{sec:general} 
allow us to extend findings of Amour and Ben-Artzi \cite{AB} and Gilding et al. \cite{Gilding} 
on the existence of solutions to nonlinear heat equations without lateral boundaries. 
Section \ref{sec:pdeDiri} is devoted to BSDEs with random terminal times and 
parabolic PDEs with lateral boundary conditions of Dirichlet type.
Finally, Section \ref{sec:pdeNeu} discusses BSDEs based on reflected forward processes and their 
relation to parabolic PDEs with lateral boundary conditions of Neumann type.

\setcounter{equation}{0}
\section{General BSDEs with terminal conditions that have bounded Malliavin derivative}
\label{sec:general}

Let $(W_t)_{0 \le t \le T}$ be an $n$-dimensional Brownian motion on a probability
space $(\Omega, {\cal F}, \p)$ and $\xi$ an ${\cal F}_T$-measurable random variable, where
$({\cal F}_t)_{0 \le t \le T}$ is the augmented filtration generated by $W$. The driver $f$ of the 
BSDE \eqref{bsde} is assumed to be a
function from $[0,T] \times \Omega \times \mathbb{R} \times \mathbb{R}^n$ to $\mathbb{R}$
that is measurable with respect to ${\cal P} \otimes {\cal B}(\mathbb{R}) \otimes {\cal B}(\mathbb{R}^n)$,
where ${\cal P}$ is the predictable sigma-algebra on $[0,T] \times \Omega$.
As usual, we identify random variables that are equal $\p$-almost surely and accordingly,
understand equalities and inequalities between them in the
$\p$-almost sure sense. The Euclidean norm on $\mathbb{R}^d$ is denoted
by $|.|$, and $xy$ stands for $\sum_{i=1}^d x_i y_i$, $x,y \in \mathbb{R}^d$.
We work with the following

\begin{definition}
A solution of the BSDE \eqref{bsde} is a pair $(Y_t,Z_t)_{0 \leq t \leq T}$
of predictable processes taking values in $\mathbb{R} \times \mathbb{R}^n$
such that $\int_0^T \left(|f(t,Y_t,Z_t)| + |Z_t|^2\right) dt < \infty$ and $\eqref{bsde}$ holds for all $0 \le t \le T$.
\end{definition}

\noindent
For $p \in [1,\infty]$, we denote
\begin{itemize}
\item $\mathbb{S}^p(\bbR^{d}) :=$ the space of $\mathbb{R}^d$-valued continuous 
adapted processes $X$ satisfying
\begin{align*}
\|X\|_{\mathbb{S}^p}: &= \norm{\sup_{0 \le t \le T}|X_{t}|}_{L^p}
<\infty
\end{align*}
where processes $X,Y$ are identified if $\|X-Y\|_{\mathbb{S}^p} = 0$.

\item $\bbH^{p}(\bbR^{d}) :=$ the space of $\mathbb{R}^d$-valued predictable processes
$X$ satisfying
\beas
&& \|X\|_{\mathbb{H}^p} := \norm{\brak{\int_{0}^{T} |X_{t}|^2
dt}^{1/2}}_{L^p} <\infty \quad \mbox{if } p<\infty \mbox{ and }\\
&& \|X\|_{\mathbb{H}^\infty} := \esssup_{(t,\omega)\in[0,T]\times\Omega} |X_{t}(\omega)|<\infty
\quad \mbox{if } p = \infty,
\eeas
where processes $X,Y$ are identified if $\|X-Y\|_{\mathbb{H}^p} = 0$.
\end{itemize}

\noindent
$(f,\xi)$ are said to be $p$-standard parameters if they satisfy the following three conditions:
\begin{itemize}
\item[(S1)] $\xi \in L^p({\cal F}_T)$
\item[(S2)]
$|f(t,y,z) - f(t,y',z')| \le L( |y-y'| + |z-z'|)$ for a constant $L \in \mathbb{R}_+$
\item[(S3)] $f(.,0,0) \in \mathbb{H}^p(\mathbb{R})$.
\end{itemize}

It is shown in Section 5 of El Karoui et al. \cite{ElKaroui} that for all $p \in (1, \infty)$, 
a BSDE of the form \eqref{bsde} with p-standard parameters has a unique solution 
$(Y,Z) \in \mathbb{S}^p(\bbR) \times \bbH^p(\bbR^n)$.

We recall that ${\cal H} := L^2([0,T]; \mathbb{R}^n)$ is a Hilbert space with
scalar product $\ang{h_1, h_2} := \int_0^T h_1(t)h_2(t) dt$, and the mapping
$h \mapsto \int_0^T h(t) dW_t$ is a Hilbert space isomorphism between
${\cal H}$ and the first Wiener chaos of $W$. The corresponding Malliavin derivative
of a Malliavin differentiable random variable $\xi$ is an $n$-dimensional
stochastic process $D_t \xi$, $0 \le t \le T$, whose components we denote by $D^i_t \xi$, $i =1, \dots, n$. 
The Sobolev space $\mathbb{D}^{1,2}$ is defined as the closure of the class of smooth random variables $\xi$ with
respect to the norm $\|\xi\|_{1,2} := \brak{\E{\xi^2 + \int_0^T |D_t \xi|^2 dt}}^{1/2}$; see Nualart \cite{Nualart}.
$\bbL^{1,2}_a(\bbR^{d})$ denotes the space of $\bbR^{d}$-valued progressively
measurable processes $X$  satisfying
\begin{itemize}
\item[{\rm (i)}] $X_t \in (\bbD^{1,2})^d$ for almost all $t$
\item[{\rm (ii)}] $(t,\omega) \mapsto DX_t(\omega) \in (L^2[0,T])^{d \times n}$ admits a 
progressively measurable version
\item[{\rm (iii)}]
$\|X\|_{\mathbb{L}^{1,2}_a}^2 := \|X\|_{\mathbb{H}^2} 
+ \norm{\brak{\int_{0}^{T}\int_{0}^{T}|D_r X_t|^2 dr dt}^{1/2}}_{L^2} < \infty,$
\end{itemize}
where processes $X,Y$ are identified if $\|X-Y\|_{\mathbb{L}^{1,2}_a} = 0$.

Now consider the conditions:

\begin{itemize}\label{hypo}
\item[(A1)] The terminal condition $\xi$ is in $\mathbb{D}^{1,2}$ and
there exist constants $A_i \in \mathbb{R}_+$ such that
$|D^i_t \xi| \le A_i$ $dt \otimes d\p \text{-a.e.}$ for all $i=1, \dots, n$.

\item[(A2)] 
There exist a constant $B \in \mathbb{R}_+$ and a nondecreasing function 
$\rho : \mathbb{R}_+ \to \mathbb{R}_+$ such that
$$
|f(t,y,z)-f(t,y',z)| \le B |y-y'| \quad \mbox{and} \quad |f(t,y,z)-f(t,y,z')|  \le \rho(|z| \vee |z'|) |z-z'|
$$
for all $t \in[0,T]$, $y,y' \in \mathbb{R}$ and $z,z' \in \mathbb{R}^n$.

\item[(A3)] $f(.,0,0) \in \mathbb{H}^4(\mathbb{R})$ and there exist
Borel-measurable functions $q_i : [0,T]\rightarrow \bbR_+$ satisfying $\int_{0}^{T} q^2_i(t) dt <\infty$
such that for every pair $(y,z) \in \mathbb{R} \times \mathbb{R}^n$ with
$$
|z| \le Q := \sqrt{\sum_{i=1}^n \brak{A_i + \int_0^T q_i(t) e^{-B(T-t)} dt}^2} \; e^{BT},
$$
one has $f(\cdot,y,z) \in \bbL^{1,2}_a(\bbR)$ and
$|D^i_r f(t,y,z)| \le q_i(t)$ $dr \otimes d\p \text{-a.e.}$ for all $i =1, \dots, n$.

\item[(A4)] 
For a.a. $r \in[0,T]$, there exists a non-negative process $K_{r.}$ in $\mathbb{H}^4(\mathbb{R})$
such that 
$$
\int_0^T \|K_{r.}\|_{\mathbb{H}^4}^4 dr < \infty \quad \mbox{and} \quad
|D_r f(t,y,z) - D_r f(t,y',z')| \le K_{rt}(|y-y'|+|z-z'|)
$$
for all $t \in [0,T]$, $y,y' \in \mathbb{R}$ and $z,z'\in \mathbb{R}^n$ satisfying $|z|,|z'|\le Q$.
\end{itemize}

Then one has the following

\begin{theorem} \label{thmmain}
If {\rm (A1)--(A4)} hold, then the BSDE \eqref{bsde} has a unique solution $(Y,Z)$ 
in $\mathbb{S}^4(\mathbb{R}) \times \mathbb{H}^{\infty}(\mathbb{R}^n)$, and
for all $i=1, \dots, n$,
$$
|Z^i_t| \le \left(A_i+ \int_t^T q_i(s) e^{-B(T-s)} ds \right) e^{B(T-t)} \quad dt \otimes d\mathbb{P}\mbox{-a.e.}
$$
\end{theorem}

\begin{Remark} \label{rmk1} 
If for a.a. $r \in [0,T]$, the process 
$K_{r.}$ in (A4) is bounded, the condition $f(.,0,0) \in \mathbb{H}^4(\mathbb{R})$ can be dropped from
(A3). Then the statement of Theorem \ref{thmmain} still holds, except that 
$Y$ is in $\mathbb{S}^2(\mathbb{R})$ instead of $\mathbb{S}^4(\mathbb{R})$.
This is due to the fact that in this case, $f(.,0,0) \in \mathbb{H}^4(\mathbb{R})$
is not needed in Proposition \ref{propsandwich} below; see Remark \ref{rmk2}.
\end{Remark}

We first prove Theorem \ref{thmmain} under the following 
stronger versions of conditions (A2)--(A4):

\begin{itemize}
\item[(A2')] $f(t,y,z)$ is continuously differentiable in $(y,z)$ and
there exist constants $B, \rho \in \mathbb{R}_+$ such that 
$$
|\partial_{y}f(t,y,z)| \le B, \quad |\partial_{z}f(t,y,z)| \le \rho
$$
for all $t \in[0,T]$, $y \in \mathbb{R}$ and $z \in \mathbb{R}^n$.
\item[(A3')]
Condition (A3) holds for all $(y,z) \in \mathbb{R} \times \mathbb{R}^n$.
\item[(A4')]
Condition (A4) holds for all $t \in[0,T]$, $y,y' \in \mathbb{R}$ and $z,z' \in \mathbb{R}^n$.
\end{itemize}

\begin{proposition} \label{propsandwich}
If {\rm (A1), (A2'), (A3'), (A4')} hold, then
the BSDE \eqref{bsde} has a unique solution $(Y,Z)$ in 
$\mathbb{S}^4(\mathbb{R}) \times \mathbb{H}^4(\mathbb{R}^n)$, and 
\be \label{Zbound}
|Z^i_t| \le \left(A_i+ \int_t^T q_i(s) e^{-B(T-s)} ds \right)e^{B(T-t)} \quad dt \otimes d\p
\mbox{-a.e.}
\ee
\end{proposition}

\begin{proof}
By Lemma \ref{lemmalp} below, condition (A1) implies $\mathbb{E}|\xi|^p < \infty$ for all $p \in \mathbb{R}_+$.
So it follows from Theorem 5.1 and Proposition 5.3 of El Karoui et al. \cite{ElKaroui} that the 
BSDE \eqref{bsde} has a unique 
solution $(Y,Z)$ in $\mathbb{S}^4(\mathbb{R}) \times \mathbb{H}^4(\mathbb{R}^n)$. Moreover,
$(Y,Z) \in \mathbb{L}^{1,2}_a(\mathbb{R}^{n+1})$, and for fixed $i=1, \dots, n$, 
$$(D^i_r Y_t, D^i_r Z_t) = (U^r_t, V^r_t) \; \; dr \otimes dt \otimes d\p\mbox{-a.e.} \quad \mbox{and} \quad
Z^i_t = U^t_t \; \; dt \otimes d\p\mbox{-a.e.},
$$ 
where 
$$
U^r_t = 0, \; V^r_t = 0, \quad 0 \le t < r \le T,
$$
and for each fixed $r$, $(U^r_t,V^r_t)_{r \le t \le T}$ is the unique pair in 
$\mathbb{S}^2(\mathbb{R}) \times \mathbb{H}^2(\mathbb{R}^n)$ solving the BSDE
\be \label{bsdedydz}
U^r_t= D_r^i \xi + \int_t^T [\partial_y f(s,Y_s,Z_s)
U^r_s + \partial_z f(s,Y_s,Z_s) V^r_s + D_r^i f(s,Y_s,Z_s)]ds -\int_t^T V^r_s dW_s.
\ee
Since \eqref{bsdedydz} and the two BSDEs 
\begin{align}
\label{bsde1}
& \overline{U}_t = A_i+\int_t^T \left(B|\overline{U}_s|+ \rho |\overline{V}_s|+ q_i(s) \right)ds
-\int_{t}^{T} \overline{V}_s dW_{s}\\
\label{bsde2}
& \underline{U}_t =  - A_i - \int_{t}^{T}\left(B|\underline{U}_s|+ \rho |\underline{V}_s|
+ q_i(s) \right)ds - \int_{t}^{T} \underline{V}_s dW_{s}
\end{align}
have 2-standard parameters, one obtains from the comparison result, Theorem 2.2 in 
El Karoui et al. \cite{ElKaroui}, that $\underline{U}_t \le U^r_t \le \overline{U}_t$ for all $t \in [0,T]$. 
But the solutions to \eqref{bsde1} and \eqref{bsde2} are given by
$$
\overline{U}_t = - \underline{U}_t =
\brak{A_i+\int_{t}^{T} q_i(s) e^{-B(T-s)}ds} e^{B(T-t)}, \quad \overline{V}_t = \underline{V}_t =0.
$$
This shows \eqref{Zbound}.
\end{proof}

\begin{lemma} \label{lemmalp}
If $\xi$ satisfies {\rm (A1)}, then $\mathbb{E} |\xi|^p < \infty$ for all $p \in [1, \infty)$.
\end{lemma}

\begin{proof}
If $\xi$ satisfies (A1), it is square-integrable.
By the Clark--Ocone formula, one can represent $\xi$ as 
$\xi= \E{\xi} + \int_0^T \E{D_t \xi | \scF_t} dW_t$. Applying the
Burkholder--Davis--Gundy inequality to the martingale $M_t = \int_0^t \E{D_t \xi | \scF_s} dW_s$, 
one obtains a constant $c_p \in \mathbb{R}_+$ such that
$$
\E{\sup_{0 \le t \le T} |M_t|^p} \le c_p \E{\brak{\int_0^T |\E{D_t\xi|\scF_t}|^2 dt}^{p/2}} < \infty,
$$
which proves the lemma.
\end{proof}

\begin{Remark} \label{rmk2}
If for a.a. $r \in [0,T]$, the process $K_{r.}$ in (A4') is bounded, 
Proposition \ref{propsandwich} still holds if the condition $f(.,0,0)\in \bbH^4(\bbR)$
is dropped from (A3') except that then,
$(Y,Z)$ is in $\mathbb{S}^2(\mathbb{R}) \times \mathbb{H}^2(\mathbb{R}^n)$
and not necessarily in $\mathbb{S}^4(\mathbb{R}) \times \mathbb{H}^4(\mathbb{R}^n)$.
This is true because in this case, the proof of Proposition 5.3 in El Karoui et al. \cite{ElKaroui} still works
without the assumption $f(.,0,0) \in \bbH^4(\bbR)$ with the difference that it yields a solution 
$(Y,Z)$ of the BSDE \eqref{bsde} in $\mathbb{S}^2(\mathbb{R}) \times \mathbb{H}^2(\mathbb{R}^n)$
instead of $\mathbb{S}^4(\mathbb{R}) \times \mathbb{H}^4(\mathbb{R}^n)$.
\end{Remark}

To derive Theorem \ref{thmmain} from Proposition \ref{propsandwich}, we need the following result,
which is Proposition 5.1 of El Karoui et al. \cite{ElKaroui} in the special case of 
a Brownian filtration and $p=2$.

\begin{proposition} {\bf (El Karoui et al., 1997)} \label{propdiff}
For every $L \in \mathbb{R}_+$ there exist constants $\mu,\nu > 0$ satisfying the following:
If $T \le \mu$, then for all $2$-standard parameters $(f^i,\xi^i)$, $i=1,2$, such that 
$f^1$ fulfills the Lipschitz condition {\rm (S2)} with Lipschitz constant $L$, 
the BSDE solutions $(Y^i,Z^i)$ corresponding to $(f^i, \xi^i)$ satisfy
\beas
\norm{Y^1-Y^2}_{\bbS^2}^2 + \norm{Z^1-Z^2}_{\bbH^2}^2 \le \nu \, 
\E{\abs{\xi^1-\xi^2}^2 + \int_0^T \brak{f^1(t,Y^2_t,Z^2_t) - f^2(t,Y^2_t,Z^2_t)}^2dt}.
\eeas
\end{proposition}

\noindent
{\bf Proof of Theorem \ref{thmmain}}\\
Define 
$$
\hat{f}(t,y,z) = \left\{
\begin{array}{cc}
f(t,y,z) & \mbox{ if } |z| \le Q\\
f(t,y, Q z/|z|) & \mbox{ if } |z| > Q
\end{array} 
\right..
$$
Then $(\hat{f},\xi)$ are $4$-standard parameters. So the corresponding BSDE has a unique solution 
$(Y,Z)$ in $\mathbb{S}^4(\mathbb{R}) \times \mathbb{H}^4(\mathbb{R}^n)$.
Denote $x = (y,z) \in \mathbb{R}^{n+1}$ and let $\beta \in C^{\infty}_c (\mathbb{R}^{n+1})$ be the 
mollifier
$$
\beta(x) := 
\left\{
\begin{array}{ll}
\lambda \exp \brak{- \frac{1}{1- |x|^2}} &\text { if } |x|<1\\
0 &\text{ otherwise}
\end{array} \right.,
$$
where the constant $\lambda \in \mathbb{R}_+$ is chosen so that $\int_{\mathbb{R}^{n+1}} \beta(x) dx = 1$.
Set $\beta^m(x) := m^{n+1} \beta(m x)$, $m \in \mathbb{N} \setminus \crl{0}$, and define
$$
f^m(t,\omega,x):= \int_{\bbR^{n+1}} \hat{f}(t,\omega,x') \beta^m(x - x') dx'.
$$
Then all $f^m$ satisfy (A2')--(A4').
Therefore, one obtains from Proposition \ref{propsandwich} that there exist unique solutions $(Y^m,Z^m)$ in
$\mathbb{S}^4(\mathbb{R}) \times \mathbb{H}^4(\mathbb{R})$ to the BSDEs corresponding to
$(f^m,\xi)$, and $|Z_t^{m,i}| \le a_i(t) :=  (A_i + \int_t^T q_i(s) e^{-B(T-s)} ds) e^{B(T-t)}$. 
Since $\hat{f}$ satisfies the Lipschitz condition (S2) 
for some constant $L \in \mathbb{R}_+$, one can choose 
constants $\mu,\nu > 0$ such that the statement of Proposition \ref{propdiff} holds. This gives
$$ 
\norm{Y-Y^{m}}_{\bbS^2,[T-\mu,T]}^2+\norm{Z-Z^{m}}_{\bbH^2, [T-\mu,T]}^2
\le \nu \, \E{ \int_{T-\mu}^T \brak{\hat{f}(t,Y^{m}_t,Z^{m}_t) - f^{m}(t,Y^{m}_t,Z^{m}_t)}^2 dt}.
$$
Since $\abs{\hat{f}-f^m} \to 0$ uniformly in $(t,\omega,y,z)$ as $m \to \infty$,
one obtains $\E{(Y_{T-\mu} - Y^m_{T-\mu})^2} \to 0$ and $|Z^i_t| \le a_i(t)$ for $T- \mu \le t \le T$. 
Proposition \ref{propdiff} applied on the interval $[T-2\mu, T- \mu]$ yields
\beas
&& \norm{Y-Y^{m}}_{\bbS^2,[T-2\mu,T-\mu]}^2+\norm{Z-Z^{m}}_{\bbH^2, [T-2\mu,T-\mu]}^2\\
&\le& \nu \, \E{(Y_{T-\mu} - Y^m_{T-\mu})^2 + \int_{T-2\mu}^{T-\mu} 
\brak{\hat{f}(t,Y^{m}_t,Z^{m}_t) -f^{m}(t,Y^{m}_t,Z^{m}_t)}^2 dt}.
\eeas
So $\E{(Y_{T-2\mu} - Y^m_{T-2\mu})^2} \to 0$ 
and $|Z^i_t| \le a_i(t)$ for $T- 2\mu \le t \le T-\mu$. By repeating this argument,
one gets $|Z^i(t)| \le a_i(t)$ for all $t \in [0,T]$. It follows that $(Y,Z)$ is also a solution 
of the BSDE \eqref{bsde} with parameters $(f,\xi)$.

Finally, if $(\tilde{Y}, \tilde{Z})$ is another solution in 
$\mathbb{S}^4(\mathbb{R}) \times \mathbb{H}^{\infty}(\mathbb{R}^n)$ corresponding to $(f,\xi)$,
it must be equal to $(Y,Z)$ since both solve the BSDE \eqref{bsde} with a $4$-standard 
driver $\tilde{f}$ that coincides with $f$ for $|z| \le \tilde{Q}$, where 
$\tilde{Q} \in \mathbb{R}_+$ is a bound on $Z$ and $\tilde{Z}$.
\qed

\bigskip

In the following corollary, we assume that the terminal condition $\xi$ is bounded
and has bounded Malliavin derivative. This allows us to relax some of the 
assumptions of Theorem \ref{thmmain} on the driver $f$. 
The precise conditions we need are the following:

\begin{itemize}
\item[(B1)] 
$\xi$ satisfies (A1) and there exists a constant $C \in \mathbb{R}_+$ such that $|\xi| \le C$.

\item[(B2)] 
There exist constants $B,D \in \mathbb{R}_+$ and a nondecreasing function
$\rho : \mathbb{R}_+ \to \mathbb{R}_+$ such that 
\beas
&& |f(t,y,z)-f(t,y',z)| \le B|y-y'|\\
&& |f(t,y,z)-f(t,y,z')| \le \rho(|z| \vee |z'|) |z-z'|\\
&& |f(t,y,z)| \le D(1+|y|) + \rho(|z|) |z|
\eeas
for all $t \in[0,T]$, $y,y' \in \mathbb{R}$ with $|y|, |y'| \le R := (C+1) e^{DT}-1$ and all
$z,z' \in \mathbb{R}^n$.

\item[(B3)]
Condition (A3) holds for all $(y,z) \in \mathbb{R} \times \mathbb{R}^n$ such that
$|y| \le R$ and 
$$|z| \le Q := \sqrt{\sum_{i=1}^n \brak{A_i + \int_0^T q_i(t) e^{-B(T-t)} dt}^2} \; e^{BT}.$$

\item[(B4)] 
Condition (A4) holds for all $t \in [0,T]$, $y,y' \in \mathbb{R}$ and $z,z'\in \mathbb{R}^n$ such that
$|y|, |y'| \le R$ and $|z|, |z'| \le Q$.
\end{itemize}

\begin{corollary} \label{cormain}
Assume {\rm (B1)--(B4)}. Then the BSDE \eqref{bsde} has a unique solution $(Y,Z)$ 
in $\mathbb{S}^{\infty}(\mathbb{R}) \times \mathbb{H}^{\infty}(\mathbb{R}^n)$, and
\begin{align*}
|Y_t| & \le (C+1) e^{D(T-t)} -1 \quad \text{ for all } t\in[0,T]\\
|Z^i_t| &\le \left(A_i+ \int_t^T q_i(s) e^{-B(T-s)}ds\right)
e^{B(T-t)} \quad dt \otimes d\mathbb{P}\mbox{-a.e.} \quad \mbox{for all } i = 1, \dots, n.
\end{align*}
\end{corollary}

\begin{proof}
Consider the following three BSDEs
\bea
\label{bsdea}
Y_t &=& \xi+\int_t^T \hat{f}(s,Y_s,Z_s)ds - \int_t^T Z_s dW_s\\
\label{bsdeb}
\overline{Y}_t &=& C + \int_t^T \overline{f} (s, \overline{Y}_s, \overline{Z}_s)ds-\int_t^T \overline{Z}_sdW_s\\
\label{bsdec}
\underline{Y}_t&=& - C +\int_t^T \underline{f}(s,\underline{Y}_s,\underline{Z}_s)ds-\int_t^T \underline{Z}_sdW_s,
\eea
where $\hat{f}(t,y,z) := f(t,\tilde{y}, \tilde{z})$ for 
$$
\tilde{y} := \left\{ \begin{array}{cc}
y & \mbox{ if } |y| \le R\\
R y/|y| & \mbox{ if } |y| > R
\end{array} 
\right. \quad \mbox{and} \quad
\tilde{z} := \left\{ \begin{array}{cc}
z & \mbox{ if } |z| \le Q\\
Q z/|z| & \mbox{ if } |z| > Q
\end{array} 
\right.,
$$
$\overline{f} (t,y,z) := D(1+|y|) + \rho(Q) |z|$ and
$\underline{f}(t,y,z)  := - \overline{f} (t,y,z)$.
$\hat{f}$ satisfies (A2)--(A4) and has the following two properties:
\begin{itemize}
\item[1)] $\hat{f}(t,y,z) = f(t,y,z)$ for all $(t,y,z)$ such that $|y| \le R$ and $|z| \le Q$
\item[2)] $\underline{f}(t,y,z) \le \hat{f}(t,y,z)\le \overline{f}(t,y,z)$ for all $(t,y,z)$.
\end{itemize}
It follows from Theorem \ref{thmmain} that \eqref{bsdea} has a unique solution $(Y,Z)$ in 
$\mathbb{S}^4(\mathbb{R}) \times \mathbb{H}^{\infty}(\mathbb{R}^n)$, and
$$
|Z^i_t| \le \brak{A_i+ \int_t^T q_i(s) e^{-B(T-s)} ds} e^{B(T-t)}.
$$
Moreover, one obtains from Theorem 2.2 in El Karoui et al. \cite{ElKaroui} that
$$
\underline{Y}_t \le Y_t \le \overline{Y}_t, \quad 0 \le t \le T,
$$
and it can easily be checked that
$$
\overline{Y}_t = - \underline{Y}_t = (C+1) e^{D(T-t)} -1, \quad \overline{Z}_t = \underline{Z}_t =0.
$$
This gives $|Y_t| \le (C+1) e^{D(T-t)} -1 \le R$. So $(Y,Z)$ solves the 
BSDE \eqref{bsde} with parameters $(f, \xi)$.

To conclude the proof, assume that $(\tilde{Y}, \tilde{Z})$ is another solution in
$\mathbb{S}^{\infty}(\mathbb{R}) \times \mathbb{H}^{\infty}(\mathbb{R}^n)$. Let $\tilde{Q} \in \mathbb{R}_+$ be
a bound on $\tilde{Z}$ and assume
$$
t^* := \sup \crl{s \in [0,T] : \p[|\tilde{Y}_s| \ge R] > 0} > 0.
$$
On $[t^*,T]$, $\tilde{Y}$ is bounded by $R$, and hence, $(\tilde{Y}, \tilde{Z})$ 
is equal to $(Y,Z)$ since both solve the BSDE \eqref{bsde} with a $4$-standard 
driver $\tilde{f}$ that coincides with $f$ for $|y| \le R$ and $|z| \le Q \vee \tilde{Q}$. 
In particular, $|\tilde{Y}_{t^*}| \le (C+1) e^{D(T-t^*)} -1 < R$. It follows that 
there exists an $\varepsilon > 0$ such that
$$
|\tilde{Y}_t| = |\bbE_t \tilde{Y}_{t^*} + \int_t^{t^*} \bbE_t f(s,\tilde{Y}_s,\tilde{Z}_s) ds|
\le (C+1) e^{D(T-t^*)} -1 + (t^*-t)[D(1+R) + \rho(\tilde{Q}) \tilde{Q}] < R
$$
for all $t \in [t^*-\varepsilon, t^*]$, a contradiction to the definition of $t^*$.
This shows that $t^* = 0$ and $(\tilde{Y}, \tilde{Z}) = (Y,Z)$.
\end{proof}

\setcounter{equation}{0}
\section{Lipschitz continuity and bounded Malliavin derivatives}
\label{sec:Lipschitz}

In this section we show that terminal conditions $\xi$ which
are Lipschitz continuous in the underlying Brownian motion $W$
are Malliavin differentiable with bounded Malliavin derivative.
On the other hand, we give an example of a terminal condition with 
bounded Malliavin derivative that is not Lipschitz continuous in $W$.
This shows that condition (A1) is weaker than Lipschitz continuity in $W$.

\begin{definition}
We denote the space of all continuous functions from $[0,T]$ to $\mathbb{R}^n$ starting from $0$ 
by $C^n_0[0,T]$ and call a random variable $\xi$ Lipschitz continuous in the Brownian motion $W$ 
with constants $A_1, \dots, A_n \in \mathbb{R}_+$ if $\xi = \varphi(W)$ for a function 
$\varphi : C^n_0[0,T] \to \mathbb{R}$ satisfying 
\be \label{Lip}
|\varphi(v) - \varphi(w)| \le \sum_{i=1}^n A_i \sup_{0 \le t \le T} |v^i(t) - w^i(t)|.
\ee
\end{definition}

\begin{proposition} \label{lipschitzprop}
Let $\xi$ be Lipschitz continuous in $W$ with constants $A_1, \dots, A_n \in \mathbb{R}_+$.
Then $\xi \in \bbD^{1,2}$ and $|D^i_t \xi| \le A_i$ $dt \otimes d\p$-a.e. for all $i=1, \dots, n$.
\end{proposition}

\begin{proof}
Assume $\xi$ is of the form $\varphi(W)$ for a function $\varphi$ satisfying \eqref{Lip}.
For $m \in \mathbb{N}$, set $t^m_j := jT/m$, $j=0, \dots, m$, and define the mapping
$
l^m:\left\{x = \left(x_j \right)_{j=1}^m : x_j \in \bbR^n \right\} \to C^n_0[0,T]
$ by 
$$ 
l^m_0(x) := 0 \quad \mbox{and} \quad
l^m_t(x) :=  x_1 + \dots + x_{j-1} + \frac{t-t^m_{j-1}}{T/m} x_j \quad \mbox{for } t^m_{j-1} < t \le t^m_j.
$$
Set $\xi^m := \varphi \circ l^m(\Delta W_{t^m_1},\dots, \Delta W_{t^m_m})$. 
For every $p \in [2, \infty)$, there exists a constant $b_p \in \mathbb{R}_+$ such that 
\beas
&& \bbE |\xi-\xi^m|^p \le b_p \, \bbE\sup_{0 \le t \le T} 
|W^1_t-l^{m,1}_t(\Delta W_{t^m_1},\dots, \Delta W_{t^m_m})|^p\\ 
&\le& b_p \, \bbE \max_{j=1,\dots, m} \sup_{t^m_{j-1} < t \le t^m_j} \left|W^1_t- W^1_{t^m_{j-1}} - 
\frac{t-t^m_{j-1}}{T/m} \Delta W^1_{t^m_j}\right|^p\\
&\le& b_p m \bbE \sup_{0 < t \le T/m} \left|W^1_t-\frac{tW^1_{T/m}}{T/m}\right|^p,
\eeas
where for the last inequality, we used that $W$ has stationary increments. It follows that
\beas
&& \norm{\xi-\xi^m}_p \le (b_p m)^{1/p} \norm{\sup_{0 < t \le T/m} \left|W^1_t-\frac{t W^1_{T/m}}{T/m}\right|}_p \le
(b_p m)^{1/p} \brak{\norm{\sup_{0 < t \le T/m} \left|W^1_t \right|}_p + \norm{W^1_{T/m}}_p}\\
&\le& (b_p m)^{1/p} c_p  \norm{W^1_{T/m}}_p \le (b_p m)^{1/p} d_p \sqrt{T/m},
\eeas
where $c_p$ and $d_p$ are constants depending on $p$, and the third inequality follows from 
Doob's maximal inequality. For $p>2$ the last term goes to $0$ as
$m \to\infty$. This shows that $\xi^m \to\xi$ in $L^p$ for all $p \in (2,\infty)$ and therefore also in $L^2$.
 
Note that for $x,y \in \mathbb{R}^{mn}$, 
\be \label{phiLip}
|\varphi \circ l^m(x) - \varphi \circ l^m(y)| \le \sum_{i,j} A_i |x^i_j - y^i_j|.
\ee
Let $\beta \in C^\infty_c(\bbR^{mn})$ be the mollifier
$$
\beta(x) := \left\{
\begin{array}{ll}
\lambda \exp \brak{- \frac{1}{1- |x|^2}} &\text { if } |x|<1\\
0 &\text{ otherwise}
\end{array} \right.,
$$
where $\lambda$ is a constant so that $\int_{\mathbb{R}^{m n}} \beta(x) dx = 1$.
Set $\beta^m(x) := m^{mn} \beta(m x)$ and define
$$
\varphi^m(x) := \int_{\mathbb{R}^{mn}} \varphi \circ l^m(y) \beta^m(x-y) dy, \quad
\tilde{\xi}^m := \varphi^m(\Delta W_{t^m_1}, \dots, \Delta W_{t^m_m}).
$$
By Proposition 1.2.3 of Nualart \cite{Nualart}, one has 
$$
D^i \tilde{\xi}^m = \sum_{j=1}^m \frac{\partial}{\partial x^i_j} \varphi^m(\Delta W_{t^m_1}, \dots, \Delta W_{t^m_m})
1_{(t^m_{j-1}, t^m_j]}.
$$
But it follows from \eqref{phiLip} that 
$\abs{\frac{\partial}{\partial x^i_j} \varphi^m(x)} \le A_i \quad \mbox{for all } i,j$.
So $|D^i_t \tilde{\xi}^m| \le A_i$ $dt \otimes d\p$-a.e.
Moreover $\tilde{\xi}^m \to \xi$ in $L^2$. Hence, one obtains from
Lemma 1.2.3 of Nualart \cite{Nualart} that $\xi$ is in $\mathbb{D}^{1,2}$ and 
$D \tilde{\xi}^m \to D \xi$ in the weak topology of $L^2(\Omega;{\cal H})$.
This implies that $|D^i_t \xi| \le A_i$ $dt \otimes d\p$-a.e.
\end{proof}

In the following example we construct a random variable with bounded Malliavin derivative that 
is not Lipschitz in the underlying Brownian motion.

\begin{Example}
Assume $T = n = 1$. Define
$$
g(t) := \sum_{k=1}^\infty (-1)^{k-1}2^k 1_{\crl{1-2^{1-k} < t \le 1-2^{-k}}}, \quad h(t) :=\int_0^t g(s) ds,
$$
and set
$$
\xi :=\int_0^1 h(t)dW_t.
$$
Then $\xi \in\bbD^{1,2}$ and $D \xi = h$ is bounded by $1$.

On the other hand, it follows from integration by parts that
$$
\int_0^{1-2^{-2k}} h(t) dW_t = - \int_0^{1 - 2^{-2k}} g(t) W_t dt \quad \mbox{for all } k \ge 1.
$$
Therefore,
$$
\xi = - \lim_{k \to \infty} \int_0^{1 - 2^{-2k}} g(t) W_t dt, 
$$ 
which shows that $\xi$ cannot be of the form $\xi = \varphi(W)$ for a Lipschitz 
continuous function $\varphi : C_0[0,1] \to \mathbb{R}$.
\end{Example}

\setcounter{equation}{0}
\section{Markovian BSDEs and semilinear parabolic PDEs}
\label{sec:pde}

For $(t,x) \in [0,T] \times \mathbb{R}^m$, consider an SDE of the form
\be
\label{sde}
X^{t,x}_{s} = x+\int_{t}^{s}b(r,X^{t,x}_{r})dr+\int_{t}^{s}\sigma(r)dW_r \quad t \le s \le T,
\ee
where $b : [0,T] \times \mathbb{R}^m \to \mathbb{R}^m$ 
and $\sigma : [0,T] \to \mathbb{R}^{m \times n}$ are Borel measurable functions 
for which there exist constants $E,F \in \mathbb{R}_+$ 
such that for all $t \in [0,T]$, $x,x' \in \mathbb{R}^m$ and $i,j$,
\bea
\label{sigmaL} |\sigma_{ij}(t)| &\le& E\\
\label{bbound} |b_i(t,x)| &\le& F(1+ \max_k |x_k|)\\
\label{bL} |b_i(t,x)-b_i(t,x')| &\le& F \max_k |x_k-x'_k|.
\eea

Denote $W^t_s := W_s - W_t$, $s \in [t,T]$, and let $({\cal F}^t_s)_{s \in [t,T]}$ be the filtration 
generated by $W^t$. By $\mathbb{S}^p_t(\mathbb{R}^d)$ we denote the space of all
$\mathbb{R}^d$-valued continuous $({\cal F}^t_s)$-adapted processes with finite $\mathbb{S}^p$-norm on $[t,T]$,
and by $\mathbb{H}^p_t(\mathbb{R}^d)$ the space of all $\mathbb{R}^d$-valued
$({\cal F}^t_s)$-predictable processes with finite $\mathbb{H}^p$-norm on $[t,T]$.
Analogously, we denote by $\mathbb{D}^{1,2}_t$ and $\mathbb{L}^{1,2}_{a,t}$ the spaces
$\mathbb{D}^{1,2}$ and $\mathbb{L}^{1,2}_a$ with respect to $(W^t_s)_{s \in [t,T]}$.

Under \eqref{sigmaL}--\eqref{bL} the SDE \eqref{sde} has a unique strong solution in 
$\mathbb{S}^2_t (\mathbb{R}^m)$; see for instance, Karatzas and Shreve \cite{KS}.
A Markovian BSDE based on $X^{t,x}$ is of the form
\be \label{markovian}
Y^{t,x}_{s} = h(X^{t,x}_{T})+ \int_{s}^{T} g(r,X^{t,x}_{r},Y^{t,x}_{r},Z^{t,x}_{r})dr-\int_{s}^{T}Z^{t,x}_{r}dW_{r}
\ee
for measurable functions $g : [0,T]\times\bbR^m \times\bbR\times\bbR^n \to \bbR$
and $h : \mathbb{R}^m \to \mathbb{R}$.

It is well-known that if $g$ is sufficiently regular in $(r,x)$ and 
Lipschitz in $(y,z)$, $u(t,x) = Y^{t,x}_t$ is a viscosity solution 
of the parabolic PDE with terminal condition
\be \label{pde}
u_t (t,x)+ {\cal L}_{(t,x)} u(t,x)+ g(t,x,u(t,x),\nabla u\sigma(t,x)) = 0, \quad
u(T,x) = h(x), 
\ee
where
$$
{\cal L}_{(t,x)} :=\frac{1}{2} \sum_{i,j} (\sigma \sigma^T)_{ij}(t)\partial_{x_{i}}\partial_{x_{j}}+\sum_{i}b_{i}(t,x)\partial_{x_{i}};
$$
see El Karoui et al. \cite{ElKaroui}.
Since Theorem \ref{thmmain} and Corollary \ref{cormain} give bounds 
on solutions of BSDEs, we can generalize this relationship between BSDEs and PDEs to the case 
where $g$ is non-Lipschitz in $(y,z)$. To do that we require $g$ and $h$ to satisfy the following conditions:

\begin{itemize}
\item[(C1)] There exists a constant $A \in \mathbb{R}_+$ such that 
$|h(x)-h(x')| \le A \max_i |x_i-x'_i|$ for all $x,x' \in \mathbb{R}^m$.

\item[(C2)] 
There exist a constant $B \in \mathbb{R}_+$ and a nondecreasing function 
$\rho : \mathbb{R}_+ \to \mathbb{R}_+$ such that
$$
|g(t,x,y,z)-g(t,x,y',z)| \le B|y-y'| \mbox{ and }
|g(t,x,y,z)-g(t,x,y,z')| \le \rho\brak{|z|\vee|z'|}|z-z'|
$$
for all $t\in[0,T]$, $x \in \mathbb{R}^m$, $y,y'\in\bbR$ and $z,z'\in\bbR^n$.

\item[(C3)]
$\int_0^T g(t,0,0,0)^2 dt < \infty$ and there exists a constant $G \in \mathbb{R}_+$ such that 
$$
|g(t,x,y,z)- g(t,x',y,z)| \le G \max_i |x_i-x'_i|
$$
for all $t \in [0,T]$, $x,x' \in \mathbb{R}^m$, $y \in \mathbb{R}$ and 
$z \in \mathbb{R}^n$ with
$$
|z| \le N := \sqrt{n} \brak{A + \frac{1-e^{-BT}}{B}G} E e^{(B+F)T}.
$$
\item[(C4)]  
There exists a constant $H \in \mathbb{R}_+$ such that 
$$
\abs{g(t,x,y,z) -g(t,x',y,z)-g(t,x,y',z')+g(t,x',y',z')} \le H \max_i |x_i-x'_i|\brak{|y-y'|+|z-z'|}$$
for all $t \in [0,T]$, $x,x' \in \mathbb{R}^m$, $y,y' \in \mathbb{R}$ and 
$z,z' \in \mathbb{R}^n$ with $|z|,|z'|\le N$.
\end{itemize}

\begin{proposition}\label{propm1}
Assume {\rm (C1)--(C4)}. Then for every $(t,x) \in [0,T] \times \mathbb{R}^m$,
the Markovian BSDE \eqref{markovian} has a unique solution $(Y^{t,x},Z^{t,x})$ in 
$\mathbb{S}^2_t(\mathbb{R}) \times \mathbb{H}^{\infty}_t(\mathbb{R}^n)$, and
$$
|Z^{t,x,i}_s| \le \brak{A + \frac{1-e^{-B(T-s)}}{B}G} E e^{B(T-s)} e^{F(T-t)} \quad \mbox{$ds \otimes d\p$-a.e.}
\quad \mbox{for all } i = 1, \dots, n.
$$ 
\end{proposition}

\begin{proof}
If we can show that the BSDE \eqref{markovian} satisfies (A1) with $A_i = A E e^{F(T-t)}$,
(A2), (A3) with $q_i \equiv G Ee^{F(T-t)}$ but without $f(.,0,0) \in \mathbb{H}^4(\mathbb{R})$ and
(A4) with a constant $K$, then the proposition follows from Theorem \ref{thmmain} and Remark \ref{rmk1}.

(A2) is a direct consequence of (C2).
By Lemma \ref{lemmaX} below, $X^{t,x}_s$ is in $(\mathbb{D}^{1,2}_t)^m$ for all $t \le s \le T$ and
$|D^i_r X^{t,x,j}_s| \le E e^{F(T-t)}$ $dr \otimes d\p$-a.e. for all $i$ and $j$.
It follows from the Lipschitz condition (C1) and Proposition 1.2.4 of Nualart \cite{Nualart} that 
$h(X^{t,x}_T)$ is in $\mathbb{D}^{1,2}_t$ and for all $i =1, \dots, n$, there exists an
$m$-dimensional random vector $\Lambda$ satisfying
$$
D^i_r h(X^{t,x}_T) = \sum_{j=1}^m \Lambda^j D^i_r X^{t,x,j}_T \quad \mbox{and} \quad
\sum_{j=1}^m |\Lambda^j| \le A.
$$
This shows that the terminal condition $\xi = h(X^{t,x}_T)$ satisfies (A1) with 
$A_i = AE e^{F(T-t)}$.
Analogously, it follows from (C3) that for every pair $(y,z)$ such that $|z| \le N$, 
$g(.,X^{t,x}_{.}, y,z)$ belongs to $\mathbb{L}^{1,2}_{a,t}$ and 
$|D^i_r g(s,X^{t,x}_s,y,z)| \le G E e^{F(T-t)}$. So (A3) holds with $q_i \equiv G E e^{F(T-t)}$.
The same argument applied to 
$$
\tilde{g}(s,x,y,y',z,z') = g(s,x,y,z) - g(s,x,y',z')
$$
gives $|D^i_r g(s,X^{t,x}_s,y,z) - D^i_r g(s,X^{t,x}_s,y',z')| \le
H E e^{F(T-t)}(|y-y'| + |z-z'|)$ for all $y,y' \in \mathbb{R}$ and 
$z,z' \in \mathbb{R}^n$ with $|z|, |z'| \le N$. This shows that (A4) holds 
with a constant $K$. 
\end{proof}

\begin{lemma}\label{lemmaX}
For all $0 \le t \le s \le T$ and $x \in \mathbb{R}^m$, $X^{t,x}_s$ is in $(\mathbb{D}^{1,2}_t)^m$ and 
$$
|D^i_r X^{t,x,j}_s| \le E e^{F(T-t)} \quad \mbox{$dr \otimes d\p$-a.e.}
\quad \mbox{for all } i=1, \dots, n \mbox{ and } j =1, \dots, m.
$$
\end{lemma}

\begin{proof}
It follows from Theorem 2.2.1 of Nualart \cite{Nualart} that $X^{t,x}_s$ is in $(\mathbb{D}^{1,2}_t)^m$.
Moreover, one obtains from the Lipschitz condition \eqref{bL} and
Proposition 1.2.4 of Nualart \cite{Nualart} that there exists an $\mathbb{R}^{m \times m}$-valued process
$\Lambda$ such that
$$
D^i b_j(s,X^{t,x}_s) = \sum_{l=1}^m \Lambda^{jl}_s D^i X^{t,x,l}_s 
 \quad \mbox{and} \quad \sum_{l=1}^m |\Lambda^{jl}_s| \le F.
$$
It follows that $\int_t^s b_j(u,X^{t,x}_u)du \in \bbD^{1,2}_t$ with
$$
\abs{D_r^i \int_t^s b_j(u,X^{t,x}_u) du} \le \int_t^s \abs{D_r^i b_j(u,X^{t,x}_u)} du
\le F \int_t^s \max_l |D_r^i X^{t,x,l}_u| du.
$$
Moreover, 
$
|D^i \int_t^s \sum_{l=1}^n \sigma_{jl}(u) dW^l_u| = |\sigma_{ji} 1_{[t, s]}| \le E.
$
Therefore,
$$
\max_j |D_r^iX^{t,x,j}_s| \leq E  + F \int_t^s \max_j |D_r^i X^{t,x,j}_u| du,
$$
and one obtains from Gronwall's lemma that
$
|D_r^iX^{t,x,j}_s| \le E e^{F(s-t)}$ $dr \otimes d\p$-a.e.
\end{proof}

If the function $h$ is bounded, one can relax some of the assumptions of Propositon \ref{propm1} on $g$
as follows:

\begin{itemize}
\item [(D1)] The function $h$ satisfies (C1) and is bounded by a constant $C \in \mathbb{R}_+$.

\item[(D2)] There exist constants $B,D \in\bbR_+$ and a nondecreasing
function $\rho : \bbR_+ \to \bbR_+$ such that 
\beas
&&|g(t,x,y,z) - g(t,x,y',z)| \le B|y-y'|\\
&&|g(t,x,y,z) - g(t,x,y,z')| \le \rho(|z| \vee |z'|) |z-z'|\\
&&|g(t,x,y,z)| \le D(1+|y|) + \rho(|z|) |z|
\eeas
for all $t \in[0,T]$, $x \in \mathbb{R}^m$, $y\in\bbR$ with 
$|y|,|y'| \le R := (C+1) e^{DT} - 1$ and all $z,z' \in\bbR^n$. 

\item[(D3)] Condition (C3) holds for all 
for all $t \in [0,T]$, $x,x' \in \mathbb{R}^m$, $y \in \mathbb{R}$ and 
$z \in \mathbb{R}^n$ such that $|y| \le R$ and $|z| \le N$.

\item[(D4)] Condition (C4) holds for all $t \in [0,T]$, $x,x' \in \mathbb{R}^m$, $y,y' \in \mathbb{R}$ and 
$z,z' \in \mathbb{R}^n$ such that $|y|,|y'| \le R$ and $|z|,|z'| \le N$.
\end{itemize}

\begin{proposition} \label{propm2}
Assume {\rm (D1)--(D4)}. Then for all $(t,x) \in [0,T] \times \mathbb{R}^m$,
the Markovian BSDE \eqref{markovian} has a unique solution $(Y^{t,x},Z^{t,x})$ 
in $\mathbb{S}^\infty_t(\mathbb{R}) \times \mathbb{H}^{\infty}_t(\mathbb{R}^n)$, and
\begin{align*}
|Y^{t,x}_s| & \le (C+1) e^{D(T-s)} -1 \quad \text{ for all } s \in[t,T]\\
|Z^{t,x,i}_s| & \le \brak{A + \frac{1-e^{-B(T-s)}}{B}G} E e^{B(T-s)} e^{F(T-t)} 
\quad \mbox{$ds \otimes d\p$-a.e.} \quad \mbox{for all } i = 1, \dots, n.
\end{align*}
\end{proposition}

\begin{proof}
(D1)--(D4) imply (B1)--(B4). Therefore, the proposition follows from  Corollary \ref{cormain}
like Proposition \ref{propm1} follows from Theorem \ref{thmmain}.
\end{proof}

\begin{corollary} \label{corpde} 
If the assumptions of Proposition \ref{propm1} or Proposition
\ref{propm2} hold, then the PDE \eqref{pde} has a viscosity 
solution $u$ such that for all $(t,x) \in [0,T] \times \mathbb{R}^m$,
$u(s,X^{t,x}_s) = Y^{t,x}_s$, $t \le s \le T$, where $X^{t,x}$ and $Y^{t,x}$ are solutions 
of \eqref{sde} and \eqref{markovian}, respectively. 
\end{corollary}

\begin{proof}
If the assumptions of Proposition \ref{propm1} hold, 
the BSDE \eqref{markovian} has for all $(t,x) \in [0,T] \times \mathbb{R}^m$
a solution $(Y^{t,x},Z^{t,x})$ such that $Z^{t,x}$ is bounded by $N$.
So $(Y^{t,x},Z^{t,x})$ also solves \eqref{markovian} if $g$ is replaced by a function $\tilde{g}$ that 
agrees with $g$ for $|z| \le N$ and is Lipschitz in $(x,y,z)$. It
follows from Theorem 4.3 of Pardoux and Peng \cite{PP92}
that $u(t,x) := Y^{t,x}_t$ is a viscosity solution of \eqref{pde} such that 
$u(s,X^{t,x}_s) = Y^{t,x}_s$, $t \le s \le T$. 

Under the assumptions of Proposition \ref{propm2}, 
the BSDE \eqref{markovian} has a solution $(Y^{t,x},Z^{t,x})$ such that 
$Y^{t,x}$ is bounded by $(C+1) e^{DT} -1$ and $Z^{t,x}$ by $N$. Then
$(Y^{t,x},Z^{t,x})$ still solves \eqref{markovian} if $g$ is replaced by a function $\tilde{g}$ that 
is Lipschitz in $(x,y,z)$ and agrees with $g$ for $|y| \le (C+1) e^{DT}-1$ and $|z| \le N$.
As above it follows that $u(t,x) := Y^{t,x}_t$ is a viscosity solution of \eqref{pde} such that 
$u(s,X^{t,x}_s) = Y^{t,x}_s$, $t \le s \le T$. 
\end{proof}

\begin{corollary} \label{corunique}
Assume the conditions of Proposition \ref{propm2} hold and set $u(t,x):=Y^{t,x}_t$. If
for every $L\in\bbR_+$, there exists a constant $\gamma_L \in \bbR$ and a continuous
function $\delta_L:\bbR_+\to\bbR_+$ with $\delta_L(0)=0$ such that
\be  \label{uni}
\begin{aligned}
g(t,x,y',v \sigma(t))-g(t,x,y,v \sigma(t)) \ge \gamma_L(y-y')\\ 
|g(t,x,y,v \sigma(t))-g(t,x',y,v \sigma(t))| \le \delta_L(|x-x'|(1+|v|))
\end{aligned}
\ee
for all $(t,x,x') \in[0,T]\times\bbR^m \times \bbR^m$, $-L \le y' \le y \le L$ and $v \in\bbR^m$,
then $u$ is the unique bounded viscosity solution of the PDE \eqref{pde}.
\end{corollary}

\begin{proof}
This follows from Section 4.2 of Ishii and Lions \cite{Ishii}.
\end{proof}

Under appropriate assumptions on the coefficients $b,\sigma, g$ and $h$, the PDE \eqref{pde} 
has a unique classical solution.

\begin{corollary} \label{corclassical} 
Assume $\int_0^T g^2(t,0,0,0) dt < \infty$, $b$ only depends on $x$, $\sigma$ is a constant
and $b,g,h$ are all $C^3$ in $(x,y,z)$. Then one has the following:
\begin{itemize} 
\item[{\rm a)}] If {\rm (C1)--(C2)} hold and there exists a constant 
$G \in \mathbb{R}_+$ such that $|\frac{\partial}{\partial x_i} g(t,x,y,z)| \le G$ for all 
$i$, $t,x,y$ and $z$ with 
$$
|z| \le N := \sqrt{n} \brak{A + \frac{1-e^{-BT}}{B} G} E e^{(B+F)T},
$$
and $b, g, h$ have bounded derivatives of first, second and third order 
in $(x,y,z)$ on the set $\crl{(t,x,y,z)\in [0,T]\times\bbR^m \times \bbR \times \bbR^n: |z| \le N}$,
then the PDE \eqref{pde} has a unique solution $u$ of class
$C^{1,2}$ such that $\nabla u\sigma$ is bounded, and 
$$
|\nabla u \sigma(t,x)| \le \sqrt{n} \brak{A + \frac{1-e^{-B(T-t)}}{B}G} E e^{(B+F)(T-t)}
\quad \mbox{for all } (t,x) \in [0,T] \times \mathbb{R}^m.
$$
\item[{\rm b)}] 
If {\rm (D1)--(D2)} hold and there exists a constant 
$G \in \mathbb{R}_+$ such that $|\frac{\partial}{\partial x_i} g(t,x,y,z)| \le G$ for all 
$i$, $t,x,y$ and $z$ with 
$$
|y| \le (C+1) e^{DT} - 1 \quad \mbox{and} \quad
|z| \le N := \sqrt{n} \brak{A + \frac{1-e^{-BT}}{B} G} E e^{(B+F)T},
$$
and $b, g, h$ have bounded derivatives of first, second and third order in $(x,y,z)$ 
on the set
$\crl{(t,x,y,z)\in [0,T] \times \bbR^m \times\bbR\times\bbR^n: |y| \le (C+1) e^{DT} - 1, |z| \le N}$,
then \eqref{pde} has a unique solution $u$ of class
$C^{1,2}$ such that $u$ and $\nabla u \sigma$ are bounded. Moreover, one has
$$
|u(t,x)| \le (C+1) e^{D(T-t)} - 1 \quad \mbox{and} \quad 
|\nabla u \sigma(t,x)| \le \sqrt{n} \brak{A + \frac{1-e^{-B(T-t)}}{B}G} E e^{(B+F)(T-t)}
$$
for all $(t,x) \in [0,T] \times \mathbb{R}^m$.
\end{itemize}
\end{corollary}

\begin{proof}
It follows from the assumptions by the mean value theorem that in case a), (C3)--(C4) are satisfied
and in case b), (D3)--(D4) hold. So one obtains from Propositions \ref{propm1} and \ref{propm2} that 
in both cases, the BSDE \eqref{markovian} has a unique solution $(Y^{t,x}, Z^{t,x})$ in 
$\mathbb{S}^2_t(\mathbb{R}) \times \mathbb{H}^{\infty}_t(\mathbb{R}^n)$. Moreover, 
$$
|Z^{t,x}| \le \sqrt{n} \brak{A + \frac{1-e^{-B(T-t)}}{B}G} E e^{(B+F)(T-t)},
$$ and in case b), $|Y^{t,x}| \le (C+1) e^{D(T-t)} - 1$. By modifying $g$ for pairs
$(y,z)$ that are not attained by $(Y^{t,x},Z^{t,x})$, one can assume that it is 
Lipschitz in $(y,z)$. Then it follows from Theorem 3.2 of
Pardoux and Peng \cite{PP92} that $u(t,x) := Y^{t,x}_t$ defines a $C^{1,2}$
solution of the PDE \eqref{pde}. By Corollary 4.1 of El Karoui et al. \cite{ElKaroui}, one has
$$
|(\nabla u\sigma)(t,x)| = |Z^{t,x}_t| \le \sqrt{n} \brak{A + \frac{1-e^{-B(T-t)}}{B}G} E e^{(B+F)(T-t)},
$$ 
and in case b), $|u(t,x)| = |Y^{t,x}_t| \le (C+1) e^{D(T-t)} - 1$.
 
Finally, let us prove uniqueness. In case a), if the PDE \eqref{pde} has another solution $v$ of class $C^{1,2}$
such that $\nabla v\sigma$ is bounded, it follows from It\^{o}'s lemma that $(\tilde{Y}^{t,x}_s, \tilde{Z}^{t,x}_s) =
(v(s,X^{t,x}_s),(\nabla v\sigma)(s,X^{t,x}_s))$ solves the
BSDE \eqref{markovian}. Boundedness of $\tilde{Z}^{t,x}$ implies that
$\tilde{Y}^{t,x}$ is in $\mathbb{S}^2_t(\mathbb{R})$. By the uniqueness result of
Propositions \ref{propm1}, one has $(Y^{t,x},Z^{t,x}) = (\tilde{Y}^{t,x}, \tilde{Z}^{t,x})$, 
and therefore, $u = v$. In case b), uniqueness follows from the same argument.
\end{proof}

As a consequence of the results in this section, one obtains 
the following corollary for PDEs with initial conditions of the form
\be \label{heat}
u_t = \triangle u + g(u,\nabla u), \quad u(0,x) = h(x),
\ee
where $u : [0,T] \times \mathbb{R}^n \to \mathbb{R}$.

\begin{corollary} \label{corsemiheat}
Consider the following conditions:
\begin{itemize} 
\item[{\rm (i)}] 
$g$ and $h$ satisfy {\rm (C1)--(C2)}.
\item[{\rm (ii)}] 
$g$ and $h$ satisfy {\rm (D1)--(D2)}.
\item[{\rm (iii)}]
For every $L \in\bbR_+$ there exists a constant $\gamma_L \in \bbR$ such that
$g(y',z)-g(y,z)\geq\gamma_L(y-y')$ for all $-L \le y' \le y \le L$ and $z\in\bbR^n$.
\item[{\rm (iv)}]
$g$ and $h$ have bounded derivatives of first, second and third order
on the set $$\crl{(x,y,z) \in \bbR^m \times \bbR \times \bbR^n: |z| \le \sqrt{n} A e^{BT}}.$$
\item[{\rm (v)}]
$g$ and $h$ have bounded derivatives of first, second and third order on the set
$$\crl{(x,y,z) \in \bbR^m \times\bbR\times\bbR^n: |y| \le (C+1) e^{DT} - 1, |z| \le \sqrt{n} A e^{BT}}.$$
\end{itemize}
Then the following hold:
\begin{itemize}
\item[{\rm a)}]
If {\rm (i)} is satisfied, the PDE \eqref{heat} has a viscosity solution $u$.
\item[{\rm b)}]
If {\rm (ii)} is satisfied, the PDE \eqref{heat} has a viscosity solution $u$ satisfying
$|u(t,x)| \le (C+1) e^{Dt} - 1$.
\item[{\rm c)}]
If {\rm (ii)} and {\rm (iii)} are satisfied, the PDE \eqref{heat} has a unique bounded
viscosity solution.
\item[{\rm d)}]
If {\rm (i)} and {\rm (iv)} are satisfied, the PDE \eqref{heat} has a unique 
$C^{1,2}$-solution with bounded gradient $\nabla u$, and 
$|\nabla u(t,x)| \le \sqrt{n} A e^{Bt}$.
\item[{\rm e)}]
If {\rm (ii)} and {\rm (v)} are satisfied, the PDE \eqref{heat} has a unique bounded
$C^{1,2}$-solution with bounded gradient $\nabla u$, and one has 
$|u(t,x)| \le (C+1) e^{Dt} - 1$ as well as $|\nabla u(t,x)| \le \sqrt{n} A e^{Bt}$.
\end{itemize}
\end{corollary} 

\begin{proof}
Set $m=n$, $b \equiv 0$ and $\sigma \equiv \sqrt{2} \, Id$.
Corollary \ref{corpde} applied to $\tilde{g}(y,z) = g(y,z/\sqrt{2})$ yields that 
under (i) or (ii) the PDE with terminal condition,
\be \label{pdev}
v_t + \Delta v + g(v, \nabla v) = 0, \quad v(T,x) = h(x),
\ee
has a viscosity solution $v : [0,T] \times \mathbb{R}^n \to \mathbb{R}$. Moreover, if
(ii) holds, one obtains from Proposition \ref{propm2} that $|v(t,x)| \le (C+1) e^{D (T-t)} - 1$. 
It follows that under both conditions, (i) and (ii), $u(t,x):=v(T-t,x)$ is a viscosity solution 
of \eqref{heat}, which in case (ii) satisfies $u(t,x) \le  (C+1) e^{D t} - 1$. This shows a) and b).
If (ii) and (iii) hold, one obtains from Corollary \ref{corunique} that $v$ is the unique 
bounded viscosity solution of \eqref{pdev}. Therefore, $u$ is the unique bounded 
viscosity solution of \eqref{heat}. This proves c). Finally, d) and e) follow from 
Corollary \ref{corclassical}.
\end{proof}

\begin{Remark}
In the special case $g(y,z)=\mu |z|^p$ the PDE
\eqref{heat} was studied by Amour and Ben-Artzi \cite{AB} as well as Gilding et al. \cite{Gilding}.
Amour and Ben-Artzi \cite{AB} proved the existence and
uniqueness of a classical solution for $\mu \neq 0$, $p > 1$ and $h$ a bounded $C^2$ function
with bounded derivatives of first and second order.
Gilding et al. \cite{Gilding} proved the existence and uniqueness of a classical solution 
for $\mu=1$, $p > 0$ and $h$ a continuous bounded function. Equation \eqref{heat} is more 
general, but for the existence of a viscosity solution we need $g$ to be locally Lipschitz in $z$.
To obtain a classical solution we have to assume that $g$ and $h$ are $C^3$.
\end{Remark}

\setcounter{equation}{0}
\section{BSDEs with random terminal times and parabolic PDEs with 
lateral Dirichlet boundary conditions}
\label{sec:pdeDiri}

\subsection{BSDEs with random terminal times}
Let $\tau \le T$ be a stopping time and $\xi$ an $\cal F_\tau$-measurable random variable.

\begin{definition}
We say an $\bbR\times\bbR^n$-valued predictable process $(Y,Z)$ solves the BSDE
with random terminal time,
\begin{equation} \label{bsdert}
Y_t = \xi + \int_{t\wedge \tau}^{\tau}f(s,Y_s,Z_s)ds-\int_{t\wedge \tau}^{\tau}Z_s dW_s,
\end{equation}
if $\int_0^{\tau} (|f(t,Y_t,Z_t)| dt + |Z_t|^2) dt < \infty$, $Z_t = 0$ for $t>\tau$ and
\eqref{bsdert} is satisfied for all $0 \le t \le T$.
\end{definition}

Suppose that for every $\omega \in \Omega$, the ODE
\begin{equation} \label{ode}
y_t(\omega) =\xi(\omega)-\int_{\tau(\omega)}^t f(s,\omega,y_s(\omega),0)ds, \quad t \in [\tau(\omega),T],
\end{equation}
has a unique solution $y(\omega)$,
and set $\hat{\xi}(\omega) := y_T(\omega)$. Note that $1_{\crl{\tau \le t}} y_t$ is adapted, and in the special case
$f(t,y,0)=0$, $t > \tau$, one has $\xi = \hat{\xi}$. 

\begin{proposition} \label{proprt}
Assume $\hat \xi$ satisfies {\rm (A1)} and $f$ fulfills 
{\rm (A2)--(A4)}. Then the BSDE \eqref{bsdert} has a unique solution 
$(Y,Z)$ in $\mathbb{S}^4(\bbR)\times\bbH^\infty(\bbR^n)$, and
\be \label{rtZbound}
|Z^i_t|\le \brak{A_i+\int_t^Tq_i(s)e^{-B(T-s)}ds}e^{B(T-t)}\quad dt\otimes d \bbP\text{-a.e.}
\quad \mbox{for all } i=1, ..., n.
\ee
\end{proposition}

\begin{proof}
If $\hat{\xi}$ satisfies {\rm (A1)} and $f$ fulfills {\rm (A2)--(A4)}, it 
follows from Theorem \ref{thmmain} that the BSDE 
$$
\hat{Y}_t=\hat{\xi}+\int_{t}^{T}f(s,\hat{Y}_s,\hat{Z}_s)ds-\int_{t}^{T} \hat{Z}_s dW_s
$$
has a unique solution $(\hat{Y},\hat{Z})$ in $\mathbb{S}^4(\mathbb{R}) \times \mathbb{H}^\infty(\mathbb{R}^n)$,
and $\hat{Z}$ satisfies the bound \eqref{rtZbound}. Let 
$$Q := \sqrt{\sum_{i=1}^n \brak{A_i + \int_0^T q_i(t) e^{-B(T-t)} dt}^2} \; e^{BT},
$$
and notice that $(\hat{Y}, \hat{Z})$ also solves the BSDE 
\be \label{bsdehat}
\hat{Y}_t = \hat{\xi} + \int_t^T \hat{f} (s,\hat{Y}_s,\hat{Z}_s)ds - \int_{t}^{T} \hat{Z}_s dW_s,
\ee
where $\hat{f}$ is the $4$-standard driver
$$
\hat{f}(t,y,z) = \left\{
\begin{array}{cc}
f(t,y,z) & \mbox{ if } |z| \le Q\\
f(t,y, Q z/|z|) & \mbox{ if } |z| > Q
\end{array} 
\right..
$$
By Theorem 3.4 of Darling and Pardoux \cite{DP}, the BSDE with random terminal time,
\be \label{bsdehatrt}
Y_t = \xi + \int_{t \wedge \tau}^{\tau} \hat{f} (s,Y_s,Z_s)ds - \int_{t \wedge \tau}^{\tau} Z_s dW_s,
\ee
has a unique solution $(Y,Z)$ in $\mathbb{S}^2(\mathbb{R}) \times \mathbb{H}^2(\mathbb{R}^n)$.
Now, note that the pair $(\tilde{Y}, \tilde{Z})$ given by $\tilde{Y}_t := Y_t 1_{\crl{t \le \tau}}
+ y_t 1_{\crl{\tau < t}}$ and $\tilde{Z}_t := Z_t 1_{\crl{t \le \tau}}$ is in 
$\mathbb{S}^2(\mathbb{R}) \times \mathbb{H}^2(\mathbb{R}^n)$ and solves the BSDE 
\eqref{bsdehat}. But since \eqref{bsdehat} can only have one solution in $\mathbb{S}^2(\mathbb{R}) \times \mathbb{H}^2(\mathbb{R}^n)$, one has $(\tilde{Y}, \tilde{Z}) = (\hat{Y}, \hat{Z})$. In particular, $(Y,Z)$ belongs to 
$\mathbb{S}^4(\mathbb{R}) \times \mathbb{H}^\infty(\mathbb{R}^n)$, and $Z$ satisfies the bound 
\eqref{rtZbound}. It follows that $(Y,Z)$ solves the BSDE \eqref{bsdert}.

Finally, if $(Y',Z')$ is another solution in $\mathbb{S}^4(\mathbb{R}) \times \mathbb{H}^\infty(\mathbb{R}^n)$
it must be equal to $(Y,Z)$ since both solve the BSDE \eqref{bsdehatrt} for a $4$-standard driver $f'$ 
that coincides with $f$ for $|z| \le Q'$, where $Q' \in \mathbb{R}_+$ is a bound for $Z$ and $Z'$.
\end{proof}

\begin{proposition} \label{proprt2}
If $\xi$ is bounded by a constant $C \in \mathbb{R}_+$,
$\hat{\xi}$ satisfies {\rm (A1)} and $f$ fulfills {\rm(B2)--(B4)} with 
$R = (C+1) e^{2DT} - 1$ instead of $R = (C+1) e^{DT}-1$, then the BSDE \eqref{bsdert} has a
unique solution $(Y,Z)$ in $\bbS^\infty(\bbR)\times\bbH^\infty(\bbR^n)$, and
\bea
\label{Yboundrt}
|Y_t| &\le& (C+1)e^{D(T-t)}-1 \quad \mbox{for all } t \in [0,T]\\
\label{Zboundrt}
|Z^i_t| &\le& \brak{A_i+\int_t^Tq_i(s)e^{-B(T-s)}ds}e^{B(T-t)}\quad dt 
\otimes d\mathbb{P}\mbox{-a.e.} \quad \mbox{for all } i = 1, \dots, n.
\eea
\end{proposition}

\begin{proof}
By condition (B2), one has $|y_t(\omega)| \le C +\int_{\tau(\omega)}^t D(1 + |y_s(\omega)|) ds$.
So one obtains from Gronwall's lemma that $|\hat{\xi}| \le (C+1) e^{DT} - 1$. Now it follows 
from Corollary \ref{cormain} by the same arguments as in the proof 
of Proposition \ref{proprt} that the BSDE \eqref{bsdert} has a unique solution 
$(Y,Z)$ in $\bbS^\infty(\bbR)\times\bbH^\infty(\bbR^n)$ and the bound \eqref{Zboundrt} 
is satisfied. To complete the proof, notice that since one has
$Y_t = \xi$ and $Z_t = 0$ for $t > \tau$, $(Y,Z)$ satisfies the BSDE 
$$
Y_t = \xi + \int_t^T f(s,Y_s, Z_s) 1_{\crl{s \le \tau}} ds - \int_t^T Z_s dW_s.
$$
So it follows from the comparison argument 
in the proof of Corollary \ref{cormain} that \eqref{Yboundrt} holds.
\end{proof}

\subsection{Semilinear parabolic PDEs  with lateral Dirichlet boundary conditions}

Let $\scO$ be an open connected subset of $\bbR^m$.
For every pair $(t,x) \in [0,T] \times \bar\scO$, consider the SDE
$$
X^{t,x}_s = x+\int_t^s b(r,X^{t,x}_r)dr + \int_t^s \sigma(r) dW_r,
$$
where $b$ and $\sigma$ fulfill the conditions \eqref{sigmaL}--\eqref{bL}.
Define the stopping time 
$$
\tau^{t,x} := \inf \crl{s \ge t : X^{t,x}_s \notin \scO} \wedge T,
$$
and consider the BSDE with random terminal time
\be \label{bsdemrt}
Y^{t,x}_s =h(X^{t,x}_{\tau^{t,x}})+\int_{s\wedge\tau^{t,x}}^{\tau^{t,x}} g(r,X^{t,x}_r,Y^{t,x}_r,Z^{t,x}_r)dr
- \int_{s\wedge\tau^{t,x}}^{\tau^{t,x}} Z^{t,x}_rdW_r, \quad t \le s \le T,
\ee
where $h : \bar\scO\to\bbR$ and $g:[0,T] \times \bar{\scO} \times\bbR\times\bbR^n\to\bbR$.
Let $\bar{g} :[0,T] \times \bbR^m\times\bbR\times\bbR^n\to\bbR$ be an extension of $g$ such that
for every $\omega$, the ODE 
$$
y^{t,x}_s(\omega) = h(X^{t,x}_{\tau^{t,x}}(\omega))-\int_{\tau^{t,x}(\omega)}^s
\bar{g}(r,X^{t,x}_r(\omega),y^{t,x}_r(\omega),0)dr, \quad \tau^{t,x}(\omega) \le s \le T,
$$
has a unique solution $y^{t,x}(\omega)$, and set $\hat\xi^{t,x}(\omega) := y_T^{t,x}(\omega)$. 
For the following results we need the following condition:
\begin{itemize}
\item[(E)] there exist constants $A_i \in \mathbb{R}_+$ such that for all $(t,x)\in[0,T]\times\bar\scO$,
$\hat{\xi}^{t,x} \in\bbD^{1,2}$ and $|D^i_r \hat{\xi}^{t,x}| \le A_i\;dr\otimes d\bbP$-a.e.
for all $i$.
\end{itemize}

\begin{proposition}\label{pdebdp}
Assume $g$ has an extension
$\bar{g} :[0,T]\times\bbR^m\times\bbR\times\bbR^n\to\bbR$ satisfying
{\rm (E)} and {\rm (C2)--(C4)} with 
$$
N =\sqrt{\sum_i\brak{A_i+\frac{GEe^{FT}(1-e^{-BT})}{B}}^2}e^{BT}
$$
instead of $N=\sqrt{n}\brak{A+\frac{1-e^{BT}}{B}G}Ee^{(B+F)T}$.
Then, for each pair $(t,x) \in [0,T] \times \bar{\scO}$, the BSDE \eqref{bsdemrt} has a unique solution
$(Y^{t,x},Z^{t,x})$ in $\bbS^2_t(\bbR)\times\bbH^\infty_t(\bbR^n)$, and
\be \label{Zboundmrt}
|Z^{t,x,i}_s|\leq \brak{A_i+\frac{GEe^{F(T-t)}(1-e^{-B(T-s)})}{B}}e^{B(T-s)} \quad ds\otimes d\bbP\text{-a.e.}
\quad \mbox{for all } i = 1, \dots, n.
\ee
\end{proposition}

\begin{proof}
Fix $(t,x)$, and set $\xi^{t,x} := h(X^{t,x}_{\tau^{t,x}})$. By assumption (E), $\hat{\xi}^{t,x}$ satisfies condition (A1),
and it follows from the other assumptions 
like in the proof of Proposition \ref{propm1} that
$\bar{g}(s,X^{t,x}_s, Y^{t,x}_s, Z^{t,x}_s)$ fulfills (A2), (A3) with $q_i \equiv GE e^{F(T-t)}$ but without 
$\bar{g}(s,X^{t,x}_s,0,0) \in \mathbb{H}^4(\mathbb{R})$ and 
(A4) with a constant $K$. Now the proposition follows from Theorem \ref{thmmain} and Remark \ref{rmk1} 
like Proposition \ref{pdebdp} followed from Theorem \ref{thmmain}.
\end{proof}

\begin{proposition} \label{pdebdp2}
Assume $h$ is bounded by a constant $C \in \mathbb{R}_+$ and $g$ has an extension
$\bar{g} :[0,T]\times\bbR^m\times\bbR\times\bbR^n\to\bbR$ satisfying {\rm (E)} and 
{\rm (D2)--(D4)} with 
$$
N =\sqrt{\sum_i\brak{A_i+\frac{GEe^{FT}(1-e^{-BT})}{B}}^2}e^{BT}
$$
instead of $N=\sqrt{n}\brak{A+\frac{1-e^{BT}}{B}G}Ee^{(B+F)T}$ and 
$R = (C+1) e^{2DT} - 1$ instead of $R = (C+1) e^{DT}-1$. Then, for each $(t,x) \in [0,T] \times \bar{\scO}$, the BSDE \eqref{bsdemrt} has a unique solution
$(Y^{t,x},Z^{t,x})$ in $\bbS^\infty_t(\bbR)\times\bbH^\infty_t(\bbR^n)$, and
\begin{align*}
 |Y^{t,x}_s|&\leq (C+1)e^{D(T-s)}-1 \text{ for all }s\in[t,T]\\
|Z^{t,x,i}_s|&\leq
\brak{A_i+\frac{GEe^{F(T-t)}(1-e^{-B(T-s)})}{B}}e^{B(T-s)}\quad
ds \otimes d \bbP \text{-a.e.} \quad \mbox{for all } i =1, \dots, n.
\end{align*}
\end{proposition}

\begin{proof}
The result follows from Corollary \ref{cormain} like Proposition \ref{pdebdp} follows from
Theorem \ref{thmmain} and Remark \ref{rmk1}.
\end{proof}

Under appropriate assumptions, a solution to the BSDE \eqref{bsdemrt} yields a
solution to the following parabolic PDE with Dirichlet boundary conditions:
\be \label{pdeDiri}
\begin{aligned}
&u_t (t,x)+ \scL_{(t,x)}u(t,x)+g(t,x,u(t,x),(\nabla u\sigma)(t,x)) = 0 \quad \mbox{for } (t,x) \in [0,T) \times \scO\\
&u(t,x) = h(x) \quad \mbox{for } (t,x)\in[0,T] \times \partial\scO \mbox{ and } (t,x) \in \crl{T} \times \scO,
\end{aligned}
\ee
where
$$
\scL_{(t,x)} :=\frac{1}{2}\sum_{i,j}(\sigma \sigma^T)_{ij}(t)
\partial_{x_i}\partial_{x_j}+\sum_i b_i(t,x)\partial_{x_i}.
$$

The next result is a consequence of Theorem 2.2, Lemma 3.1 and Theorem 3.2 of Peng \cite{Peng91}.

\begin{theorem} {\bf (Peng, 1991)} \label{pth}
Assume the following conditions hold:
\begin{itemize}
\item[{\rm(F1)}]
$b:[0,T]\times\bbR^m\to\bbR^m$ is in $ C^{1,2}([0,T]\times\bar\scO)$,
$\sigma:[0,T]\to\bbR^{m\times n}$ is in $ C^{1}[0,T]$, and there exists a constant $\varepsilon > 0$ such that 
$\sum_{i,j}\brak{\sigma\sigma^T}_{ij}(t)v_iv_j\ge \varepsilon |v|^2$ for all $(t,v)\in[0,T]\times\bbR^m$
\item[{\rm(F2)}] $\scO$ is bounded and $\partial\scO$ is $ C^3$
\item[{\rm(F3)}] $h$ is $C^3$ and $\scL_{(t,x)}h(x)+g(T,x,h(x),\nabla h(x)\sigma(T))=0$ for $x\in\partial\scO$
\item[{\rm(F4)}] $g(t,x,y,z)$ is continuously differentiable in $(t,x,y,z) \in [0,T] \times \bar{\scO} \times 
\mathbb{R} \times \mathbb{R}^n$ with bounded derivatives.
\end{itemize} 
Then the BSDE \eqref{bsdemrt} has a unique solution $(Y^{t,x}, Z^{t,x})$ in
$\mathbb{S}^2(\mathbb{R}) \times \mathbb{H}^2(\mathbb{R}^n)$ and 
$u(t,x):=Y^{t,x}_t$ is the unique $C^{1,2}$-solution of the PDE \eqref{pdeDiri}.
\end{theorem}

By applying Proposition \ref{pdebdp2}, one can weaken condition (F4) in Theorem \ref{pth}.

\begin{corollary} \label{pdebdt}
Assume {\rm (F1)--(F3)} are satisfied, $g$ is continuously differentiable
in $(t,x,y,z) \in [0,T] \times \bar{\scO} \times \mathbb{R} \times \mathbb{R}^n$ 
and the assumptions of Proposition \ref{pdebdp2} hold. Let $(Y^{t,x},Z^{t,x})$ be the 
unique solution of the BSDE \eqref{bsdemrt} in 
$\mathbb{S}^{\infty}_t(\mathbb{R}) \times \mathbb{H}^{\infty}_t(\mathbb{R}^n)$.
Then $u(t,x):=Y^{t,x}_t$ is the unique $C^{1,2}$-solution of the PDE \eqref{pdeDiri}, and one has
\be \label{boundsunablau}
|u(t,x)| \le (C+1) e^{D(T-t)} -1, \quad
|\nabla u(t,x)| \le \frac{1}{\sqrt{\varepsilon}} \sqrt{\sum_i\brak{A_i+\frac{GEe^{F(T-t)}(1-e^{-B(T-t)})}{B}}^2}e^{B(T-t)}.
\ee
\end{corollary}

\begin{proof}
It follows from Proposition \ref{pdebdp2} that the BSDE \eqref{bsdemrt} has a unique solution $(Y^{t,x},Z^{t,x})$ 
in $\mathbb{S}^{\infty}_t(\mathbb{R}) \times \mathbb{H}^{\infty}(\mathbb{R}^n)$ with
$|Y^{t,x}_s| \le (C+1)e^{D(T-s)}-1$ and
$$
|Z^{t,x}_s| \le 
\sqrt{\sum_i \brak{A_i + \frac{G Ee^{F(T-t)}(1-e^{-B(T-t)})}{B}}^2}e^{B(T-t)} \quad ds\otimes d\bbP\text{-a.e.}
$$
By modifying $g$ for pairs $(y,z)$ that are not attained by $(Y^{t,x}, Z^{t,x})$, one can assume that 
it has bounded derivatives. Then one obtains from Theorem \ref{pth} that 
$u(t,x):=Y^{t,x}_t $ is a $C^{1,2}$-solution of the PDE \eqref{pdeDiri}. 
It can be seen in the proof of Theorem 3.2 of Peng \cite{Peng91} that 
$Z^{t,x}_t = \nabla u(t,x) \sigma(t)$. So the bounds \eqref{boundsunablau} follow from condition (F1).

If $v$ is another $C^{1,2}$-solution of \eqref{pdeDiri}, $v$ and $\nabla v$ are bounded. Moreover, 
it follows from It\^{o}'s formula that $\tilde{Y}^{t,x}_s := v(s \wedge \tau^{t,x},X^{t,x}_{s\wedge\tau^{t,x}})$,
$\tilde{Z}^{t,x}_s := \nabla v(s,X^{t,x}_{s \wedge\tau^{t,x}})\sigma(s)1_{\crl{s \le \tau}}$ solve  
the BSDE \eqref{bsdemrt}. So one obtains from the
uniqueness result of Proposition \ref{pdebdp2} that $u(t,x) = v(t,x)$.
\end{proof}

An important assumption of Propositions \ref{pdebdp} and \ref{pdebdp2} as well as 
Corollary \ref{pdebdt} is that $g$ has an extension $\bar{g}$ such that condition ${\rm(E)}$ holds. 
$X^{t,x}_{\tau^{t,x}}$ is typically not Malliavin differentiable. For instance, if 
$\tau$ is a stopping time such that $W_{\tau} \in \bbD^{1,2}$, then $\tau$ must be a constant.
Indeed, for $W_{\tau} = \int_0^{\infty} 1_{\crl{s < \tau}} dW_s \in \bbD^{1,2}$, one obtains from 
Proposition 5.3 of El Karoui et al. (1997) that $1_{\crl{s < \tau}} \in \bbD^{1,2}$ for almost 
all $s$, and therefore, by Proposition 1.2.6 of Nualart \cite{Nualart}, $\bbP[s < \tau] = 0$ or $1$. 
So to ensure that condition (E) holds, one has to require the functions 
$g$ and $h$ to be regular enough. The following lemma gives sufficient conditions
for (E).

\begin{lemma} \label{lemmacondE}
Assume $g$ has an extension
$\bar{g} :[0,T]\times\bbR^m\times\bbR\times\bbR^n\to\bbR$ satisfying $\bar{g}(t,x,y,0) = 0$
for all $t$, $x$ and $y$. Then {\rm (E)} holds if $h : \bar{\scO} \to \mathbb{R}$ is $C^2$ 
with bounded gradient and there exist measurable functions 
$\alpha: [0,T] \times h(\bar{\scO}) \to \mathbb{R}$ and
$\beta: h(\bar{\scO}) \to \mathbb{R}^n$ such that 
\begin{itemize}
\item[{\rm (i)}]
$\scL_{(t,x)} h(x) = \alpha(t,h(x))$, $(\nabla h \sigma)(x) = \beta(h(x))$ for $(t,x) \in [0,T] \times \scO$
\item[{\rm (ii)}]
$\alpha(t,h(x)) = 0$ and $\beta(h(x)) = 0$ for $x \in \partial \scO$,
\item[{\rm (iii)}]
$|\alpha(t,x) - \alpha(t,y)| \le \kappa |x-y|$, $\alpha(t,0) \le \kappa$ and $|\beta(x) - \beta(y)| \le \kappa |x-y|$
for some constant $\kappa \in \mathbb{R}_+$
\item[{\rm (iv)}]
there exist countably many values $h_1, h_2, \dots$ in $\mathbb{R}$ such that 
$\partial \scO = \bigcup_i \crl{x \in \partial \scO : h(x) = h_i}$.
\end{itemize}
\end{lemma}

\begin{proof}
Fix $(t,x) \in [0,T] \times \scO$. Since $g(s,x,y,0) = 0$, one has $\hat{\xi}^{t,x} = h(X^{t,x}_{\tau^{t,x}})$,
and by It\^{o}'s Lemma, 
\beas
h(X^{t,x}_{s \wedge \tau^{t,x}}) &=& h(x) + \int_t^{s \wedge \tau^{t,x}} \scL_{(t,x)} h(X^{t,x}_r) dr 
+ \int_t^{s \wedge \tau^{t,x}} (\nabla h \sigma)(X^{t,x}_r) dW_r\\
&=& h(x) + \int_t^s \alpha(r,h(X^{t,x}_{r \wedge \tau^{t,x}})) dr 
+ \int_t^s \beta(h(X^{t,x}_{r \wedge \tau^{t,x}})) dW_r.
\eeas
It follows from Theorem 2.2.1 of Nualart \cite{Nualart} that $h(X^{t,x}_{\tau^{t,x}})$ belongs to $\bbD^{1,2}$. 
Moreover, since the Malliavin derivative is local (see Proposition 1.3.16 in \cite{Nualart}), one has
$$
D \hat{\xi}^{t,x} = D h(X^{t,x}_{\tau^{t,x}}) =1_{\crl{\tau^{t,x} = T}} D h(X^{t,x}_T) 
+ \sum_i 1_{\crl{\tau^{t,x} < T, \, h(X^{t,x}_{\tau^{t,x}}) = h_i}} Dh_i = 1_{\crl{\tau^{t,x} = T}} D h(X^{t,x}_T),
$$
and by the chain rule, $D h(X^{t,x}_T) = \sum_j \partial_j h(X^{t,x}_T) D X^{t,x,j}_T$. So since 
$\nabla h$ is bounded, it follows from Lemma \ref{lemmaX} that $D \hat{\xi}^{t,x}$ is uniformly bounded
in $t$ and $x$.
\end{proof}

The following example describes a more concrete situation in which condition (E) holds.

\begin{Example}
Assume $X = W$ for a one-dimensional Brownian motion $W$ and $\bar{\scO} = [a,b]$.
If $g$ has an extension $\bar{g} : [0,T] \times \mathbb{R}^3 \to \mathbb{R}$ such that 
$g(t,x,y,0) = 0$ for all $t$, $x$, $y$, and $h : [a,b] \to \mathbb{R}$ is $C^2$ such that
$h'(x) = \beta(h(x))$ for a $C^2$-function $\beta :h[a,b] \to \mathbb{R}$ satisfying 
$\beta(a) = \beta(b) = 0$, then $\scL_{(t,x)} h(x) = h''(x)/2 = G'(h(x))G(h(x))/2$. So the 
conditions of Lemma \ref{lemmacondE} are fulfilled, and (E) holds.
\end{Example}

\setcounter{equation}{0}
\section{Markovian BSDEs based on reflected SDEs and parabolic PDEs with lateral
Neumann boundary conditions}
\label{sec:pdeNeu}

In this whole section, $\scO \subset\bbR^n$ is an open connected domain and 
$b: \bar{\scO} \to\bbR^n$, $\sigma: \bar{\scO} \to\bbR^{n\times n}$ are bounded Lipschitz functions. 
We assume that $\scO$ satisfies the 
uniform exterior sphere condition and uniform interior cone condition 
introduced by Saisho \cite{Saisho}. They are defined as follows:
For $y \in \partial \scO$ and $r>0$, define $\scN_{y,r}:=\crl{ v\in\bbR^n:|v|=1,
 B_r(y-rv) \cap \scO = \emptyset}$ and
$\scN_y:=\cup_{r>0} \scN_{y,r}$ where $B_r(y)$ denotes the open ball around $y$ with radius $r$.
\begin{itemize}
\item[] {\bf Uniform exterior sphere condition}\\ There exists a constant $r_0>0$ such that
$\scN_{y}=\scN_{y,r_0} \neq \emptyset$ for all $y\in\partial\scO$.
\item[] {\bf Uniform interior cone condition}\\
There exist constants $\delta>0$ and $\varepsilon\in [0,1)$ with the following property:
for every $y\in\partial\scO$, there exists a unit vector $v \in \mathbb{R}^n$ such that
$$
\crl{z\in B_{\delta}(y) :\ang{z-x,v}\ge \varepsilon|z-x|} \subset\bar\scO
\quad \mbox{for all $x\in B_\delta(y)\cap\partial \scO$.}
$$
\end{itemize}

\subsection{Reflected SDEs and Markovian BSDEs}

For every pair $(t,x) \in [0,T] \times \bar{\scO}$ we define a diffusion $X^{t,x}$ that is reflected at the 
boundary of $\scO$. Let $v(y)\in\scN_y$ be a vector field on $\partial\scO$. 
Note that if $\partial\scO$ is smooth, then $v(y)$ is the unit inward normal vector at $y$.
It is shown in Saisho \cite{Saisho} that for all $(t,x)$, there exists a unique pair 
$(X^{t,x}, L^{t,x})$ of continuous adapted processes with values in 
$\bar{\scO} \times \bbR_+$ such that for all $s \in [t,T]$,
\be \label{rfsde}
\begin{aligned}
& X^{t,x}_{s}  = x+\int_{t}^{s}b(X^{t,x}_{r})dr+\int_{t}^{s}\sigma(X^{t,x}_r) dW_{r}+\int_t^{s}v(X_r^{t,x})dL^{t,x}_r\\
& L^{t,x}_s = \int_t^{s}1_{\{X_r^{t,x}\in \partial \scO\}} dL_r^{t,x}
\quad \mbox{and} \quad L^{t,x} \text{ is nondecreasing.}
\end{aligned}
\ee
Let $g : [0,T] \times \bar\scO \times \bbR \times \bbR^n \to \bbR$ and 
$h: \bar\scO \to\bbR$ be measurable functions and consider the BSDE
\be \label{nbsde}
Y^{t,x}_{s} =h(X^{t,x}_{T})+\int_{s}^{T}g(r,X^{t,x}_{r},Y^{t,x}_{r},Z^{t,x}_{r})dr-\int_{s}^{T}Z^{t,x}_{r}dW_{r},
\quad t \le s \le T.
\ee

\begin{proposition} \label{propnbsde}
Assume there exists a constant $M \in \mathbb{R}_+$ such that for all $0 \le t \le s \le T$ and $x \in \bar{\scO}$,
\be \label{condrefl}
X_s^{t,x}\in\bbD^{1,2} \quad \mbox{and} \quad |D_r X^{t,x}_s| \le M \quad dr\otimes d\p\mbox{-a.e.}
\ee
If $g$ and $h$ satisfy {\rm (C1)--(C4)} with 
$$N =\sqrt{n}\brak{A+\frac{1-e^{BT}}{B}G}Me^{BT} \quad \mbox{instead of} \quad
N = \sqrt{n}\brak{A+\frac{1-e^{BT}}{B}G}Ee^{(B+F)T},$$ then \eqref{nbsde} has for all 
$(t,x)\in[0,T]\times \bar{\scO}$ a unique solution $(Y^{t,x},Z^{t,x})\in\bbS^2_t(\bbR)
\times\bbH^\infty_t(\bbR^n)$, and
$$
|Z^{t,x,i}_s| \le \brak{A + \frac{1-e^{-B(T-s)}}{B}G} M e^{B(T-s)} \quad \mbox{$ds \otimes d\p$-a.e.}
\quad \mbox{for all } i = 1, \dots, n.
$$ 
If $g$ and $h$ satisfy {\rm (D1)--(D4)}  with 
$$N =\sqrt{n}\brak{A+\frac{1-e^{BT}}{B}G}Me^{BT} \quad \mbox{instead of} \quad
N = \sqrt{n}\brak{A+\frac{1-e^{BT}}{B}G}Ee^{(B+F)T},$$
then \eqref{nbsde} has a unique solution 
$(Y^{t,x},Z^{t,x}) \in \bbS^\infty_t(\bbR)\times\bbH^\infty_t(\bbR^n)$, and
\beas
|Y^{t,x}_s|&\le& (C+1) e^{D(T-s)}-1\quad \text{for all }s\in[t,T]\;\; a.s.\\
|Z^{t,x,i}_s|&\le& \brak{A + \frac{1-e^{-B(T-s)}}{B}G} M e^{B(T-s)}
\quad ds \otimes d\mathbb{P}\mbox{-a.e.} \quad \mbox{for all } i = 1, \dots, n.
\eeas
\end{proposition}

\begin{proof}
If $g$ and $h$ satisfy {\rm (C1)--(C4)}, the proposition follows like Proposition \ref{propm1},
and if $g$ and $h$ fulfill {\rm (D1)--(D4)}, it follows like Proposition \ref{propm2}.
\end{proof}

\eqref{condrefl} is a crucial assumption of Proposition \ref{propnbsde}.
The following lemma gives a sufficient condition for it.



\begin{lemma} \label{lemmarefl}
Assume that $\scO$ is a convex polyhedron with nonempty interior in $\bbR^n$,
$b=0$, and $\sigma= c \,Id$ for a constant $c \in\bbR_+$. Then condition \eqref{condrefl} holds.
\end{lemma}

\begin{proof}
It follows from Theorems 2.1 and 2.2 in Dupuis and Ishii \cite{DI}
that $X^{t,x}_s$ is Lipschitz continuous in $W$ with constants $A_1, \dots, A_n$ 
independent of $t,s$ and $x$. So the statement follows from
Proposition \ref{lipschitzprop}.
\end{proof}

\subsection{Semilinear parabolic PDEs with lateral Neumann boundary conditions}

Assume that $\scO \subset \mathbb{R}^n$ is bounded and there exists a function 
$w \in C^2(\bbR^n)$ with bounded derivatives of first and second order
such that $\scO=\{w >0\}, \partial \scO=\{w=0\}, \bbR^n \setminus \bar\scO=\crl{w <0}$, 
and $|\nabla w(x)|=1$ for $x\in\partial \scO$. Then $\scO$ satisfies the uniform exterior sphere 
condition and uniform interior cone condition. So for all 
$(t,x) \in [0,T] \times \bar\scO$, there exists a unique pair of continuous adapted processes 
$(X^{t,x},L^{t,x})$ with values in $\bar \scO \times \bbR_+$ such that
\be 
\begin{aligned}
X^{t,x}_{s}& =x+\int_{t}^{s}b(X^{t,x}_{r})dr+\int_{t}^{s}\sigma(X^{t,x}_r) dW_{r}
+\int_t^{s}\nabla w(X_r^{t,x})dL^{t,x}_r\\
L^{t,x}_s&=\int_t^{s}1_{\{X_r^{t,x}\in \partial \scO\}}dL_r^{t,x} 
\quad \mbox{and} \quad L^{t,x} \text{ is nondecreasing.}\notag
\end{aligned}
\ee
If the forward process is of this form, the Markovian BSDE \eqref{nbsde}
is related to the following parabolic PDE with lateral Neumann boundary conditions:
\be \label{npde}
\begin{aligned}
& u_t(t,x)+\scL_{x}u(t,x)+g\left(t,x,u(t,x),(\nabla u\sigma)(t,x)\right)=0\quad \text{ for } (t,x)\in(0,T)\times \scO\\
& \frac{\partial u}{\partial {n}}(t,x)=0 \quad \text{ for } (t,x)\in(0,T)\times\partial \scO\quad\text{ and }\quad
u(T,x)=h(x)\quad \text{ for } x \in \bar \scO,
\end{aligned} 
\ee where 
$$
\frac{\partial}{\partial {n}} :=\sum_{i=1}^n\frac{\partial w}{\partial x_i}(x)\frac{\partial}{\partial x_i},\quad\text{ and }\quad
\scL_{x} :=\frac{1}{2}\sum_{i,j}(\sigma \sigma^{T})_{ij}(x)\partial_{x_{i}}\partial_{x_{j}}+\sum_{i}b_{i}(x)\partial_{x_{i}}.
$$

\begin{proposition} \label{nbsdepde}
Assume condition \eqref{condrefl} holds and $g, h$ satisfy {\rm (D1)--(D4)} with 
$$N =\sqrt{n}\brak{A+\frac{1-e^{BT}}{B}G}Me^{BT} \quad \mbox{instead of} \quad
N = \sqrt{n}\brak{A+\frac{1-e^{BT}}{B}G}Ee^{(B+F)T}.$$ Let
$(Y^{t,x},Z^{t,x})$ be the solution of the BSDE \eqref{nbsde}.
Then, $u(t,x):=Y^{t,x}_t$ is a viscosity solution of the PDE \eqref{npde} satisfying
$|u(t,x)| \le (C+1) e^{D(T-t)} - 1$ for all $(t,x) \in [0,T] \times \bar{\scO}$.
\end{proposition}

\begin{proof}
One can assume that $g$ is Lipschitz
in $(x,y,z)$ by modifying it for large $(x,y,z)$. Then the results of Pardoux and Zhang 
\cite{PZ} apply, and one obtains that $u(t,x):=Y^{t,x}_t$ is a viscosity solution of the PDE \eqref{npde}.
By Proposition \ref{propnbsde}, it is bounded by $(C+1) e^{D(T-t)}-1$.
\end{proof}

If one makes stronger assumptions on $\scO,b,\sigma$ and $g$, 
the viscosity solution $u$ of Proposition \ref{nbsdepde} is unique.
We denote by ${\cal S}^n$ the set of all symmetric $n \times n$-matrices
and define the function $F : [0,T]\times\bbR^n\times\bbR\times\bbR^n\times\scS^n \to \mathbb{R}$ by 
$$
F(t,x,y,v,S) := 
-\half \sum_{i,j}(\sigma\sigma^T)_{ij}(x) S_{ij} - \sum_i b_i(x)v_i -g(T-t,x,y,v\sigma(x)).
$$

\begin{proposition} \label{unbsdepde}
Assume the boundary function $w$ is $C^3$ with bounded derivatives of first, second and 
third order, $g$ is continuous in $(t,x,y,z)$ and the conditions of Proposition \ref{nbsdepde} hold. 
Moreover, suppose that
for all $L, L'\in\bbR_+$, there exist a constant $\gamma_L \in \bbR$ and a function 
$\delta_{L,L'} : \mathbb{R}_+ \to \mathbb{R}_+$ satisfying $\lim_{x \downarrow 0} 
\delta_{L,L'}(x) = 0$ such that the following two conditions hold:
\begin{itemize}
\item[{\rm (i)}]
$g(t,x,y',v\sigma(x))-g(t,x,y,v \sigma(x)) \ge \gamma_L(y-y')$
for all $(t,x)\in[0,T]\times\bar\scO$,  $-L \le y'\leq y \le L$ and $v\in\bbR^n$.
\item[{\rm (ii)}]
$$
F(t,x',y,v',S')-F(t,x,y,v,S) \le \delta_{L,L'}\brak{\eta+|x-x'|(1+|v| \vee |v'|)+\frac{|x-x'|^2}{\veps^2}}
$$
for all $\eta, \varepsilon \in (0,1]$, $t \in[0,T]$, $x,x' \in \bar\scO$, $|y| \le L$, 
$v,v'\in\bbR^n$ and $S,S'\in\scS^n$ satisfying the following three properties:
\beas
&& -\frac{L'}{\veps^2}Id \le \left( \begin{array}{ccc}
S & 0\\
0 & -S' \end{array} \right)\leq\frac{L'}{\veps^2}\left( \begin{array}{ccc}
Id & -Id\\
-Id & Id \end{array} \right)+ L'\eta Id\\
&& |v-v'|\leq L' \eta\veps(1+|v| \wedge |v'|)\\
&& |x-x'| \leq L' \eta \varepsilon.
\eeas
\end{itemize}
Let  $(Y^{t,x}, Z^{t,x})$ be the solution of the BSDE \eqref{nbsde}. Then
$u(t,x):=Y^{t,x}_t$ is the unique viscosity solution of the PDE 
\eqref{npde}.
\end{proposition}

\begin{proof}
By Proposition \ref{nbsdepde}, $u(t,x) := Y^{t,x}_t$ is a viscosity solution of \eqref{npde}.
Uniqueness follows from Theorem 3.1 of Barles \cite{Barles}.
\end{proof}

The unique viscosity solution of Proposition \ref{unbsdepde} is
actually of class $C^{1,2}$ if one strengthens the assumptions.

\begin{proposition} \label{npdeclassical}
Assume the conditions of Proposition 6.4 are satisfied and the following hold:
\begin{itemize}
\item[{\rm (i)}]
$\sigma$ is $C^{2}(\bar\scO)$ with bounded derivatives of first and second order
and there exists a constant $\varepsilon > 0$ such that
$\sum_{i,j}\brak{\sigma\sigma^T}_{ij}(x)v_iv_j\ge \varepsilon |v|^2$ for all $x,v\in\bbR^n$.
\item[{\rm (ii)}]
$b(x)v +g(t,x,y,v \sigma(x))$ is continuously differentiable in $(t,x,y,v)$
\item[{\rm (iii)}]
$h \equiv 0$.
\end{itemize}
Then the PDE \eqref{npde} has a unique $C^{1,2}$-solution $u$, and
\bea
\label{ueps}
|u(t,x)|&\le& e^{D(T-t)}-1\\
\label{nablaueps}
|\nabla u(t,x)|&\le& \sqrt{\frac{n}{\veps}}\brak{A + \frac{1-e^{-B(T-t)}}{B}G} M e^{B(T-t)}.
\eea
\end{proposition}

\begin{proof}
We can assume that $g$ is Lipschitz in $(x,y,z)$ by modifying it for large $(x,y,x)$.
Then it follows from Theorem V.7.4 of Ladyzenskaja et al. \cite{Lady} that there exists a 
$C^{1,2}$ solution. So the unique viscosity solution $u$ of Proposition \ref{unbsdepde} is $C^{1,2}$.
From Pardoux and Zhang \cite{PZ}, we know that $Y^{t,x}_s=u(s,X^{t,x}_s)$. 
Since $h \equiv 0$, one obtains from Proposition \ref{nbsdepde} that $u$ satisfies \eqref{ueps}.
Now fix $(t,x) \in (0,T) \times \scO$ and let $\alpha > 0$ be a constant such that $\crl{y \in \mathbb{R}: |y-x|\leq\alpha} \subset \scO$.
Define the stopping time
$\tau^{t,x}:=\inf\crl{s\geq t: |X^{t,x}_s-x| \ge \alpha} \wedge (t+\alpha)$.
Then $(Y^{t,x}_{s \wedge \tau^{t,x}}, Z^{t,x}_s 1_{\crl{s \le \tau^{t,x}}})$ 
and $(u(s\wedge\tau^{t,x},X^{t,x}_{s \wedge\tau^{t,x}}), (\nabla
u\sigma) (s,X^{t,x}_s)1_{\crl{s \le \tau^{t,x}}})$ are bounded solutions of the BSDE
\be \label{bsdestopped}
\tilde{Y}_s=u(\tau^{t,x},X^{t,x}_{\tau^{t,x}}) + \int_s^{t+\alpha} g(r,X^{t,x}_r,
\tilde{Y}^{t,x}_r, \tilde{Z}^{t,x}_r) 1_{\crl{s \le \tau^{t,x}}} dr-\int_s^{t+\alpha} \tilde{Z}^{t,x}_r dW_r,
\ee
on $[t,t+\alpha]$. By modifying $g$ for large $(x,y,z)$, one can assume that it is Lipschitz in 
$(x,y,z)$. Then \eqref{bsdestopped} is a standard BSDE and has a unique solution. 
Therefore, one obtains from Proposition \ref{propnbsde} that
$$
|(\nabla u \sigma)(s,X^{t,x}_s) 1_{\crl{s \le \tau^{t,x}}}| =
|Z^{t,x}_s 1_{\crl{s \le \tau^{t,x}}}| \le \sqrt{n}\brak{A + \frac{1-e^{-B(T-s)}}{B}G}
M e^{B(T-s)}\quad ds\otimes d\bbP\text{-a.e.}
$$ on $[t,t+\alpha]$, and in particular,
$$
|(\nabla
u\sigma)(t,x)|\leq \sqrt{n}\brak{A + \frac{1-e^{-B(T-t)}}{B}G}
M e^{B(T-t)},
$$ 
which by condition (i), gives the bound \eqref{nablaueps}.
\end{proof}

As a consequence of Propositions \ref{propnbsde}, \ref{nbsdepde}, 
\ref{unbsdepde} and \ref{npdeclassical} one obtains the following result for PDEs of the form:
\be \label{1dpde}
\begin{aligned}
& u_t = u_{xx} +g(u,u_x)\quad \text{on } [0,T] \times (c,d)\\
& u_x = 0 \quad \text{on } \bbR_+ \times \crl{c,d}
\quad\text{and} \quad u(0,x)=h(x) \quad \text{for } x \in  (c,d),
\end{aligned}
\ee
where $u: [0,T] \times [c,d] \to\bbR$.

\begin{corollary}
Assume $h$ satisfies {\rm(C1)} and $g$ fulfills {\rm(D2)}. 
Then \eqref{1dpde} has a viscosity solution $u$ satisfying
$$
|u(t,x)| \le \brak{\sup_{c < x < d} |h(x)|+1} e^{Dt}-1, \quad (t,x) \in [0,T] \times [c,d].
$$
Moreover, if $g$ is continuous in $y$, for every $L \in \bbR_+$, 
there exists a constant $\gamma_L \in \bbR$ such that
for all $-L \le y' \le y \le L$ and $z \in \bbR^n$, one has
\begin{equation}\label{addu}
g(y',z)-g(y,z) \ge \gamma_L(y-y')
\end{equation}
and $F(t,x,y,v,S)= \sum_{i,j} S_{ij} -g(y,v)$ satisfies condition {\rm (ii)} of Proposition \ref{unbsdepde},
then $u$ is the unique viscosity solution.
If in addition, $h \equiv 0$ and $g$ is $C^1$, then $u$ is $C^{1,2}$ and satisfies
$$
|u_x (t,x)| \le 3A e^{Bt} \quad \mbox{for all } (t,x) \in [0,T] \times [c,d].
$$
\end{corollary}

\begin{proof}
Set $b \equiv 0$, $\sigma \equiv \sqrt{2} \, Id$ and
$\tilde{g}(y,z) := g(y,z/\sqrt{2})$. Since $h$ is Lipschitz continuous and $[c,d]$ is compact, 
$h$ is bounded. Therefore, $\tilde{g}$ and $h$ satisfy {\rm (D1)--(D4)}
with $C = \sup_{c < x < d} |h(x)|$ and $G =H= 0$. So one obtains from
Proposition \ref{propnbsde} and Lemma \ref{lemmarefl} that the BSDE \eqref{nbsde} has a unique solution 
$(Y^{t,x},Z^{t,x})$ in $\bbS^{\infty}_t(\bbR) \times\bbH^\infty_t(\bbR^n)$ with
$|Y^{t,x}_s| \le (C+1) e^{D(T-s)}-1$. It can be seen from Theorems 2.1 and 2.2 of
Dupuis and Ishii \cite{DI} together with Proposition \ref{lipschitzprop} that condition \eqref{condrefl} is satisfied with 
$M = 3 \sqrt{2}$. Therefore, Proposition \ref{propnbsde} yields $|Z^{t,x}_s| \le 3 \sqrt{2} Ae^{B(T-s)}$.
By Proposition \ref{nbsdepde}, $v(t,x) := Y^{t,x}_t$ is a viscosity solution of the PDE
\beas
&& v_t + v_{xx}+g(v,v_x)=0 \quad \text{on } [0,T] \times (c,d)\\
&& v_x = 0 \quad \text{on } [0,T] \times \crl{c,d} 
\quad\text{and} \quad v(T,x)=h(x) \quad \text{for } x \in (c,d),
\eeas
satisfying $|v(t,x)| \le (C+1) e^{D(T-t)}-1$. So $u(t,x) := v(T-t,x)$ is a viscosity solution of \eqref{1dpde} 
with $|u(t,x)| \le (C+1) e^{Dt}-1$.
If $g$ is continuous in $y$, \eqref{addu} holds and
$F(t,x,y,v,S)= \sum_{i,j} S_{ij} -g(y,v)$ fulfills condition (ii) of Proposition \ref{unbsdepde},
then the conditions of Proposition \ref{unbsdepde} are satisfied. So $u$ is the unique viscosity solution.
If in addition, $h \equiv 0$ and $g$ is of class $C^1$, one obtains from Proposition \ref{npdeclassical}
that $u$ is of class $C^{1,2}$ and $|u_x(t,x)| \le 3A e^{Bt}$. 
\end{proof}


\begin{thebibliography}{25}
\bibitem{AB}{Amour, L. and Ben-Artzi, M. (1998). Global existence and
   decay for viscous Hamilton Jacobi equations. Nonlinear Analysis,
    Theory, Methods \& Applications 31(5-6), 621--628.} 

\bibitem{Barles}{Barles, G. (1999). Nonlinear Neumann boundary conditions for quasilinear 
degenerate elliptic equations and applications. J. Diff. Eq. 154(1), 191--224.}

\bibitem{BE}{Briand, P. and Elie, R. (2013). A simple constructive approach to quadratic BSDEs
with or without delay.
Stoch. Proc. Appl. 123, 2921--2939.}

\bibitem{BH}{Briand, P. and Hu, Y. (2006). BSDEs with quadratic growth and unbounded terminal value. 
Prob. Th. Rel. Fields 136(4), 604--618.}

\bibitem{BH2}{Briand, P. and Hu, Y. (2008). Quadratic BSDEs with convex generators and unbounded terminal conditions. Prob. Th. Rel. Fields 141(3-4), 543--567.}

\bibitem{CS}{Cheridito, P. and Stadje, M. (2012). Existence, minimality and approximation of solutions to BSDEs with convex drivers. Stoch. Proc. Appl. 122(4), 1540--1565.}

\bibitem{DP}{Darling, R.W.R. and Pardoux, E. (1997). Backwards SDEs with random terminal time and applications to semilinear elliptic PDE. Ann. Prob. 25(3), 1135--1159.}

\bibitem{DHB}{Delbaen, F., Hu, Y., and Bao, X. (2010). Backward SDEs with superquadratic growth. 
Prob. Th. Rel. Fields 150, 145--192.}

\bibitem{DHR}{Delbaen, F., Hu, Y., and Richou, A. (2011). On the uniqueness of solutions to quadratic BSDEs with convex generators and unbounded terminal conditions. 
Ann. Inst. Henri Poincar{\'e} Prob. et Stat. 47(2), 559--574.} 

\bibitem{DI}{Dupuis, P. and Ishii, H. (1991). On Lipschitz
continuity of the solution mapping to the Skorokhod problem, with
applications. Stochastics and Stochastic Reports, 35(1), 31--62.}

 \bibitem{ElKaroui}{El Karoui N., Peng S., and Quenez, M.C. (1997). 
Backward stochastic differential equations in finance. Math. Finance 7(1), 1--71.}

\bibitem{Gilding}{Gilding, B., Guedda, M., and Kersner, R. (2003). 
The Cauchy problem for $u_t=\triangle u + |\nabla u|^q$. 
J. Math. Analysis Appl. 284, 733--755.}

\bibitem{HNS}{Hu, Y., Nualart, D. and Song, X. (2011). Malliavin calculus for backward
stochastic differential equations and application to numerical solutions. Ann.
Appl. Prob. 21(6), 2379--2423.}

\bibitem{Ishii}{Ishii, H. and Lions, P.L. (1990) Viscosity solutions of fully
nonlinear second-order elliptic partial differential
equations. J. Diff. Eq. 83, 26--78.}

\bibitem{KS}{Karatzas, I. and Shreve, S.E. (1991). Brownian Motion and
Calculus. 2nd ed. Springer, New York.}

 \bibitem{Kobylanski}{Kobylanski, M. (2000). Backward stochastic
differential equations and partial differential equations with quadratic growth. 
Ann. Prob. 28(2), 558--602.}

\bibitem{Lady}{Lady{\v{z}}enskaja, O. A., Solonnikov, V. A., and Ural$'$ceva, N. N. (1968). Linear and Quasi-linear Equations of Parabolic Type. Translations of Mathematical Monographs, 23, 
American Mathematical Society, Providence, R.I.}



\bibitem{Nualart}{Nualart, D. (2006). The Malliavin Calculus and Related Topics. 2nd ed. Springer, Berlin.}


\bibitem{PP90}{Pardoux, E., Peng, S.(1990). Adapted solution of a backward stochastic differential equation. 
Syst. Control Lett. 14(1), 55--61.}

\bibitem{PP92}{Pardoux, E. and Peng, S. (1992). Backward stochastic
differential equations and quasilinear parabolic partial
differential equations. Lect. Notes Control Information Science 176, 200--217, Springer, Berlin.}

\bibitem{PZ}{Pardoux, E. and Zhang, S. (1998). Generalized BSDEs and nonlinear Neumann boundary value problems. Prob. Th. Rel. Fields 110, 535--558.}

\bibitem{Peng91}{Peng, S. (1991). Probabilistic interpretation for
systems of quasilinear parabolic partial differential
equations. Stochastics and Stochastics Reports 37, 61--74.}

\bibitem{Richou}{Richou, A. (2012). Markovian quadratic and superquadratic BSDEs with an unbounded terminal condition. Stoch. Proc. Appl. 122(9), 3173--3208.}

\bibitem{Saisho}{Saisho, Y. (1987). Stochastic differential equations for
multi-dimensional domain with reflecting boundary. Prob. Th. Rel. Fields, 74(3), 455--477.}
\end{thebibliography}
\end{document}